\documentclass[3p,times]{elsarticle}
\usepackage[colorlinks,linkcolor=blue,citecolor=blue]{hyperref}
\usepackage{bookmark}
\usepackage{bm}
\usepackage{amsmath,amsfonts,amsmath,amssymb,amsthm,cases}
\usepackage{algorithm}
\usepackage{algpseudocode} 
\usepackage{multirow}
\usepackage{longtable}
\usepackage{fullpage}
\usepackage{cleveref}
\usepackage{graphics,graphicx,epsfig,epstopdf,subfigure}
\usepackage{color,natbib}
\graphicspath{{jpg/}}

\newtheorem{theorem}{Theorem}[section]
\newtheorem{lemma}{Lemma}[section]

\newtheorem{remark}{Remark}[section]
\newtheorem{example}{Example}[section]

\newtheorem{assumption}{Assumption}[section]
\numberwithin{equation}{section}

\newcommand{\dd}{\,{\rm d}}

\begin{document}
	
\begin{frontmatter}
	\title{A Parareal exponential integrator finite element method for linear parabolic equations}
	\author[sjtu]{Jianguo Huang\fnref{hjgfootnote}}
	\ead{jghuang@sjtu.edu.cn}
	\address[sjtu]{School of Mathematical Sciences, and MOE-LSC, Shanghai Jiao Tong University, Shanghai 200240, China}
	\author[psu]{Yuejin Xu\fnref{xyjfootnote}}
	\ead{ymx5204@psu.edu}
	\address[psu]{Department of mathematics, the Pennsylvania State University, PA 16802, USA}
	\begin{abstract}
		In this paper, for solving a class of linear parabolic equations in rectangular domains, we have proposed an efficient Parareal exponential integrator finite element method. The proposed method first uses the finite element approximation with continuous multilinear rectangular basis function for spatial discretization, and then takes the Runge-Kutta approach accompanied with Parareal framework for time integration of the resulting semi-discrete system to produce parallel-in-time numerical solution. Under certain regularity assumptions, fully-discrete error estimates in $L^2$-norm are derived for the proposed schemes with random interpolation nodes. Moreover, a fast solver can be provided based on tensor product spectral decomposition and fast Fourier transform (FFT), since the mass and coefficient matrices of the proposed method can be simultaneously diagonalized with an orthogonal matrix. A series of numerical experiments in various dimensions are also presented to validate the theoretical results and demonstrate the excellent performance of the proposed method.
		
	\end{abstract}
\end{frontmatter}	
	
\section{Introduction}
In this paper, we are intended to study a fast solution of the following linear parabolic equation:
\begin{equation}
\label{eq1-1}
\left\{\begin{split}
&u_t = \mathcal{D}\Delta u +f(t), \quad \bm x \in \Omega, \quad t_0 \leq t \leq t_0+T,\\
&u(t_0):=u(t_0,\bm x)=u_0, \quad \bm x \in \Omega,
\end{split}
\right.
\end{equation}
where $\Omega$ is an open rectangular domain in $\mathbb{R}^d$ ($d\geq 1$), $T > 0$ is the duration time, $\mathcal{D}>0$ is the diffusion coefficient, and $f(t)$ is the reaction term for the above equation, which has been widely used in various scientific and engineering models. In this paper, we mainly discuss the problem above, which is associated with Dirichlet boundary conditions. The theoretical conclusion can also be extended to equations with periodic boundary conditions after minor modification.

In the past few decades, many researchers are devoted to developing fast solver for parabolic equation like \eqref{eq1-1}. Various methods have been applied to solving the numerical result for spatial and temporal discretization, respectively. Some common methods such as finite element method, finite difference method, finite volume method and spectral method are used for spatial discretization, while temporal discretization is mainly implemented by implicit-explicit method \cite{AscherRuuth1995}, fully-implicit scheme \cite{FePr2003}, split step method \cite{BoydJohn2001, SanzCalvo1994}, integrating factor method (IF) \cite{Lawson1967}, integrating factor Runge-Kutta method (IFRK) \cite{IGG2018, LLJF2021}, sliders method \cite{Driscoll2002},  exponential time differencing method (ETD) \cite{HochbruckOstermann2010,BorislavWill2005}, invariant energy quadratization method (IEQ) \cite{Yang2016}, scalar auxiliary variable method (SAV) \cite{ShXuYa2019} and so on. ETD method has received much attention in the past two decades since it approximates the nonlinear part of function $f$ using polynomial interpolations to obtain higher accuracy for highly oscillation problems \cite{ZhuJu2016}. Due to the great efficiency and stability of ETD method, the analysis of convergence \cite{HochbruckOstermann2005a,HochbruckOstermann2005b, DuZhu2004, KassamTrefethen2005, MohebbiDehghan2010} and modification \cite{KassamTrefethen2005, HochbruckOstermann2008, WhalenBrio2015, JuZhang2015, ZhuJu2016, HuangJu2019b,HuangJu2019a} has been widely proposed.

For temporal discretization, Parareal algorithm is a crucial method firstly proposed by Lions, Maday, Y. and Turinici in \cite{LionsMaday2001}, which solves equations in the parallel-in-time pattern to improve the calculating efficiency. Motivated by the two-level method, the whole time interval is divided into coarse and fine subintervals, attaching with coarse propagator $\mathcal{G}$ and fine propagator $\mathcal{F}$ respectively. The great advantage of Parareal method is that the numerical results obtained on coarse grids can reach to the accuracy of applying fine propagator squentially on fine grids with significant speedup. The convergence property of Parareal is discussed in \cite{GanderVandewalle2007, Guillaume2007} for scalar ODEs, showing that the algorithm is superlinear convergent on bounded time intervals, while that is linear convergent on unbounded intervals based on the convergence factor and the property of Banach spaces respectively. As for systems with complex eigenvalues, the convergence of Parareal and Parareal-Euler method are proposed in \cite{StaffRonquist2007, Wu2016} by deriving the corresponding stability function or convergence factor. What's more, the convergence property is discussed for some specific combinations of coarse propagator $\mathcal{G}$ and fine propagator $\mathcal{F}$ \cite{WuShi2009, WangWu2015, Wu2014}, which are mainly derived by constructing the convergence factors and plotting the stability regions. In recent years, with the development of Parareal method, many scholars have proposed some modified Parareal algorithm mainly in two aspects, decreasing the iteration times to speedup the convergence rate and improving the parallel efficiency for each iteration. Many researchers have tried to decrease the iteration times from different aspects, for instance, implicit time-integration schemes and Newton-like correction are used to accelerate the simulations in \cite{FarhatChandesris2003}. Some modified Parareal methods are proposed in \cite{GarridoLee2006} in the framework of two-level scheme, which regard the operators on coarse grid and fine grid as predictor and corrector respectively and can ensure to converge within a fixed number of iterations. Moreover, based on the relationship among Parareal, space-time multigrid and multiple shooting methods, some variants of Parareal are derived in \cite{GanderVandewalle2007} to accelerate the convergence. As for improving the efficiency of parallelism, multilevel extension can alleviate some of the load-balance issues of the two-level method \cite{GarridoLee2006}. A nonintrusive, optimal-scaling time-parallel method based on multigrid reduction (MGR) is also designed to obtain greater parallelism \cite{FalgoutFriedhoff2014}. To reduce the communication or synchronization among the processors, Paraexp method is proposed in \cite{GanderGuttel2013} to decouple the original problem to inhomogeneous and homogeneous subproblems, by untilizing an overlapping time-domain decomposition. Diagonalization is a crucial technique to reduce the computational cost in each iteration, which has been applied to solve the advection-diffusion equation with periodic boundary conditions in \cite{GanderLiu2020}, to develop the parallel coarse grid correction (CGC), and so on. 

As far as we know, for most numerical methods of PDEs, the spatial variables are often discretized by finite difference methods  \cite{JuZhang2015, ZhuJu2016, HuangJu2019a, HuangJu2019b} and the corresponding convergence analysis highly depends on discrete maximum principle \cite{JuLi2021} and references therein, which does not always hold for classical finite element or pseudo-spectral discretization. Moreover, most of Parareal convergence analysis is based on deriving the convergent factor, thus, the stability regions obtained are always approximate but not accurate when the convergent factor is complicated. In order to overcome the above obstacles, in this paper, for linear parabolic equation \eqref{eq1-1} in rectangular domains, we propose an efficient Parareal exponential integrator finite element method (abbreviated as PEIFE-linear) to solve large-scale problems in parallel-in-time pattern. Similar to \cite{HuangJu2022}, we first derive a semi-discrete scheme with continuous linear elements on a uniform rectangular partition and then obtain a fully-discrete scheme with explicit exponential Runge-Kutta scheme. Later, the whole time interval is divided into several subintervals and is applied with Parareal method in the temporal direction. Global error estimates measured in $L^2$-norm can be derived for PEIFE-linear method with random RK stages for the linear problem with Dirichlet boundary condition. It's worthy to note that the error estimate of PEIFE-linear method can be derived explicitly, by following the similar arguments of \cite{HochbruckOstermann2010} and \cite{GanderVandewalle2007, Guillaume2007}. Furthermore, motivated by \cite{JuZhang2015}, the resulting mass and coefficient matrices in PEIFE-linear method can be diagonalized simultaneously by the means of FFT, which leads to efficient implementation of spatial discretization for both Dirichlet and periodic boundary condition cases. The innovation of this paper is the development of parallel numerical method with rigorous global error estimates for general linear parabolic equations, which is proved to be unconditionally stable within finite iterations.

The rest of the paper is organized as follows. For solving parabolic equations \eqref{eq1-1} efficiently, we first propose an efficient and stable method named EIFE-linear and its parallel-in-time pattern method named PEIFE-linear in Section \ref{algorithm-description}. Next, we develop the global error analysis of EIFE-linear method in Section \ref{linear_theory}. Later, under certain assumptions of coarse propagator $\mathcal{G}$ and fine propagator $\mathcal{F}$, we develop the error estimate of Parareal method and derive that the convergence analysis of PEIFE-linear method can turn to the discussion of EIFE-linear method in Section \ref{PFEM_theory}. In Section \ref{numerical}, various numerical experiments are carried out to validate the theoretical results and demonstrate the excellent performance of the PEIFE-linear method. Finally, some remarks and discussion are drawn in Section \ref{conclusion}.

\section{The EIFE-linear method and PEIFE-linear method}
\label{algorithm-description}

In this section, for solving \eqref{eq1-1} with Dirichlet boundary conditions, we first propose the exponential integrator finite element method (EIFE-linear) for linear parabolic equations. Later, we apply Parareal framework to accelerate the simulation of EIFE-linear method and then develop the PEIFE-linear method.

First of all, we define some standard notations for later requirement. Given a bounded domain $G\subset \mathbb{R}^{d}$
and a non-negative integer $s$, denote the standard Sobolev spaces on $G$ by $H^s(G)$
with norm $\|\cdot\|_{s,G}$ and semi-norm $|\cdot|_{s,G}$, and denote the ${L}^2$-inner product on $G$ by $(\cdot,\,\cdot)_{G}$. $H_0^s(G)$ denotes the closure of $C_0^{\infty}(G)$ with respect to the norm $\|\cdot\|_{s,G}$.   We omit the subscript if $G=\Omega$, when there is no confusion caused. For any non-negative integer $\ell$, $\mathbb{P}_{\ell}(G)$ stands for the set of all polynomials on $G$ with the total degree at most $\ell$. Moreover, given two quantities $a$ and $b$, ``$a\lesssim{b}$" abbreviates ``$a\leq{C}b$", where the generic positive constant
$C$ is independent of the mesh size; $a\eqsim b$ is equivalent to $a\lesssim b\lesssim a$.

Now, we are ready to consider the problem \eqref{eq1-1} with homogeneous Dirichlet boundary condition and an initial configuration $u_0 \in H^2(\Omega)\cap H_0^1(\Omega)$, that is
\begin{equation}
\label{eq3-1}
\left\{\begin{split}
&u_t = \mathcal{D}\Delta u +f(t), \quad \bm x \in \Omega, \quad t_0 \leq t \leq t_0+T,\\
&u(t_0,\bm x)=u_0(\bm x), \quad \bm x \in \Omega,\\
&u(t,\bm x) = 0, \quad \bm x \in \partial \Omega,\quad t_0 \leq t \leq t_0+T, 
\end{split}
\right.
\end{equation}
where $0 < T < \infty$ is the total duration time.

\subsection{Semi-discretization in space by finite element approximation}

For simplicity, we denote $u(t)(\bm x)=u(t,\bm x)$, $\bm x \in \Omega$ in the following presentation. Suppose that $\Omega \in \mathbb{R}^d$ is a rectangular domain and is devided uniformly with the mesh size $h_i>0$ along $x_i$ direction. Associated with one-dimensional piecewise linear finite element space $V_{h_i}^i$ for direction $x_i$, we construct the finite element space $V_h$, which is a subspace of $H_0^1(\Omega)$ and can be denoted as
\begin{equation}
\label{tensor-FES}
\begin{split}
V_h:
=& {\rm span}\{ {\phi_1^{i_1}}(x_1)\cdots \phi_d^{i_d}(x_d):\ 1\le i_1\le N_1,\cdots,1\le i_d\le N_d \},
\end{split}
\end{equation}
where $h=\max\limits_{1 \leq i \leq d}h_i$, $\phi_i^j(x_i)$ is the j-th basis function of $V_h^i$, $N_i$ is the number of uniform subdomains of direction $x_i$. For the forthcoming error analysis, we assume the partition $\mathcal{T}_h$ is quasi-uniform, i.e., $h \eqsim h_i$ for all $1 \leq i \leq d$.

Using the variational framework, the above problem \eqref{eq3-1} can be converted to find $u_h \in V_h$, which satisfies that for each $v_h \in V_h$,
\begin{equation}
\label{eq3-11}
\left\{\begin{split}
& (u_{h,t},v_h)  + a(u_h,v_h) = (f(t),v_h), \quad t_0 \leq t \leq t_0+T, \\
&u_h(t_0)=P_hu_0, 
\end{split}
\right.
\end{equation}
where the bilinear operator $a(\cdot,\cdot)$ is symmetric and defined by 
\begin{equation}
\label{bilinear}
a(w,v)=\int_{\Omega}\mathcal{D}\nabla w \cdot \nabla v \dd \bm x, \quad w, v \in H_0^1(\Omega), 
\end{equation}
and $P_h:L^2(\Omega) \rightarrow V_h$ is the $L^2$-orthogonal projection operator and is stable with respect to $L^2$-norm.

According to the inverse inequality for finite elements and Riesz representation theorem, we can derive that there exists abounded linear operator $L_h:V_h\to V_h$ such that
\begin{equation}
\label{A-h}
a(w_h,v_h)= (L_h w_h, v_h), \quad \forall w_h,\ v_h\in V_h.
\end{equation}
Then the problem \eqref{eq3-11} can be reformulated as
\begin{equation}
\label{eq1-5}
\left\{
\begin{split}
&u_{h,t}+L_hu_h=P_hf, \quad \bm x \in \Omega, \quad t_0 \leq t \leq t_0+T,\\
&u_h(t_0)=P_h u_0, \quad \bm x \in \Omega.
\end{split}
\right.
\end{equation}

Until now, we have derived the semi-discrete scheme for \eqref{eq3-1}.

\subsection{Linear exponential integrator finite element method}

As for the temporal discretization, we divide the time interval $[t_0,t_0+T]$ into $N_T>0$ subintervals which are denoted as
\begin{equation*}
\tau_{n+1}=\tau_n+\Delta \tau_n, \quad n = 0, \cdots, N_T - 1,\quad\mbox{and } \Delta \tau=\max\limits_{0\le n\le N_T-1}\Delta \tau_n.
\end{equation*}

Denote $\{e^{-tL_h} \}_{t\geq 0}$ as the semigroup on $V_h$ with the infinitesimal generator $(-L_h)$. Hence, by the Duhamel formula, the solution $u_h$ to the problem \eqref{eq1-5} can be further expressed as
\begin{equation}
\label{eq1-6}
u_h(\tau_{n+1}) = e^{-\Delta \tau_nL_h }u_h(\tau_n)+\int_{0}^{\Delta\tau_n} e^{-(\Delta \tau_n - \sigma)L_h}P_hf(\tau_n+\sigma)\dd \sigma.
\end{equation}

For computing the numerical solution in \eqref{eq1-6} efficiently, we apply linear exponential Runge-Kutta method \cite{HochbruckOstermann2010} to compute the integral and obtain the approximate solution. Define $ u_h^n$ as the numerical solution obtained at time $t_n$ after temporal and spatial discretization. Then a fully-discrete scheme for solution of the linear parabolic problem \eqref{eq3-1} can be expressed as follows,
\begin{equation}
\label{eq1-7}
u_h^{n+1}=e^{-\Delta \tau_n L_h}u_h^n+\Delta \tau_n \sum\limits_{i=1}^s b_i(-\Delta \tau_n L_h)P_h f(\tau_n+c_i\Delta \tau_n), \quad n=0, \cdots, N_T-1,
\end{equation}
where the interpolation nodes $c_1, \cdots, c_s$ are $s$ different nodes selected in $[0, 1]$ and the weights are:
\begin{equation}
\label{eq3-12}
b_i(-\Delta \tau_n L_h) = \int_{0}^1 e^{-\Delta \tau_n(1-\theta)L_h}l_i(\theta)\dd \theta,
\end{equation}
and $l_i(\theta)$ are the familiar Lagrange interpolation polynomials
\[l_i(\theta):=\prod\limits_{m=1,m \neq i}^s \frac{\theta-c_m}{c_i-c_m}, \quad i = 1 , \cdots s. \]

Until now, we have proposed a fully-discrete numerical method called the exponential integrator finite element method (abbreviated as EIFE-linear) for linear parabolic equation (\ref{eq1-1}). Furthermore, since the semi-discrete scheme \eqref{eq1-5} can be regarded as an ODE system, it's worthy to note that EIFE-linear method can also be implemented in the parallel-in-time pattern. 

\subsection{Exponential integrator finite element method in parallel-in-time pattern}

Next, we design a parallel EIFE-linear method, which is combined with Parareal method and achieve significant speedup. Before applying Parareal algorithm to compute the fully-discrete scheme (\ref{eq1-7}) in parallel, we add more notations for Parareal algorithm. The whole time interval $[t_0, t_0+T]$ is divided into $N$ coarse time interval and each coarse subinterval is denoted as $[T_n, T_{n+1}]$ with time step $\Delta T_n, n=0, \cdots, N-1$. Then each coarse interval is divided into $M$ fine intervals while the time step of each fine interval $[t_{l}, t_{l+1}], l=0, \cdots, NM-1$ is $\Delta t_{l}$. It's obvious that $T_n=t_{nM}$ for $n=0, \cdots, N-1$. $\mathcal{G}$ is the less precise but cheaper scheme chosen to coarse grids while $\mathcal{F}$ is the more precise one chosen to fine grids. Taking $v$ as the starting point, denote $\mathcal{G}(v; T_n,T_{n+1}), \mathcal{H}(v;T_n, T_{n+1})$ as the results computed in the time interval $[T_n, T_{n+1}]$ by less precise method $\mathcal{G}$ and accurate method $\mathcal{H}$ respectively. $\mathcal{F}(v; t_{l}, t_{l+1})$ is the result calculated by more precise method $\mathcal{F}$ and $\mathcal{F}^M(v;T_n,T_{n+1})$ is obtained by applying successively $M$ substeps of fine propagator $\mathcal{F}$ in the subinterval $[T_n, T_{n+1}]$. To illustrate the algorithm better, define $U_{h}^{n,(k)}$ as the numerical result computed at time $T_n$ after $k$ iterations. Meanwhile, we designate the time-sequential parts of the algorithm by symbol $\ominus$, and the parallel parts by symbol $\oplus$.

Since the coarse propagator is applied to the coarse grids with larger time step, it's naturally to require coarse propagator $\mathcal{G}$ to be unconditionally stable. In the following section, we will prove that EIFE-linear method is unconditionally stable for linear problems (Section \ref{linear_theory}). Since Parareal method has been proved that once it's convergent, the approximation on the coarse grids can achieve the accuracy of using the $\mathcal{F}$ propagator step by step with the small time step $\Delta t_n$. Therefore, in order to balance the need of computational cost and numerical accuracy, we choose a less precise EIFE-linear scheme with $p$ stages as the coarse propagator $\mathcal{G}$ while a more precise one with $q$ stages as the fine propagator $\mathcal{F}$ ($p \leq q$), that is,
\begin{equation*}
\left\{\begin{split}
\mathcal{G}(v;T_n,T_{n+1})&=e^{-\Delta T_nL_h}v+\Delta T_n\sum\limits_{i=1}^p b_i(-\Delta T_nL_h)P_hf(T_n+c_i\Delta T_n),\\
\mathcal{F}(v;t_l,t_{l+1})&= e^{-\Delta t_lL_h}v+\Delta t_l\sum\limits_{i=1}^q b_i(-\Delta t_lL_h)P_hf(t_l+c_i\Delta t_l).
\end{split}
\right.
\end{equation*}
Until now, we have developed the Parareal exponential integrator finite element method for linear problems (abbreviated as PEIFE-linear) and we show the framework of PEIFE-linear method as below:
\begin{algorithm}[ht]
	\caption{\textbf{Parareal exponential integrator finite element algorithm}}
	\label{algorithm1}
	\begin{algorithmic}[1]
		\State Set $k=0$ and $U_h^{0,(0)}=U_h^0=P_hu_0$.
		\State $\ominus$ Compute by coarse propagator sequentially 
		\begin{equation*}
		U_{h}^{n+1,(0)}=\mathcal{G}(U_h^{n,(0)};T_n,T_{n+1}), \quad n=0, \cdots, N-1.
		\end{equation*}
		\While{the stopping criteria not satisfied}
		\State $\oplus$ For $n=0, \cdots, N - 1$, compute by fine propagator in parallel
		\begin{equation*}
		\overline{U}_h^{n+1,(k+1)}=\mathcal{F}^M(U_h^{n,(k)};T_n,T_{n+1}).
		\end{equation*}
		\State $\ominus$ Perform sequentially for $n=0, \cdots, N-1$
		\begin{equation}
		\label{Parareal}
		U_h^{n+1,(k+1)}=\mathcal{G}(U_h^{n,(k+1)};T_n,T_{n+1})+\overline{U}_h^{n+1,(k+1)}-\mathcal{G}(U_h^{n,(k)};T_n,T_{n+1}),
		\end{equation}
		
		with $U_h^{0,(k+1)}=U_h^0$.
		\State Set $k=k+1$.
		\EndWhile
	\end{algorithmic}
\end{algorithm}

It's worth noting that the PEIFE-linear method shown here is in abstract framework for further theoretical analysis. What's more, we can express (\ref{eq3-11}) in the matrix-vector form by representing the finite element space $V_h$ with nodal basis functions. The relative mass and coefficient matrix can be diagonalized simultaneously, whose multiplication with a vector can be implemented efficiently by the means for FFT and tensor product operations. We have also provided the details of matrix-vector implementation in Section 4 in \cite{HuangJu2022} for interested readers. Also, the fast implementation can be naturally extended to cases with periodic boundary conditions.


\section{Convergence analysis of EIFE-linear method}\label{linear_theory}

In this section, we first focus on developing the error bound of EIFE-linear method for solving the problem \eqref{eq1-1} with homogenous Dirichlet boundary conditions. It is straightforward to extend the arguments to the same problem with nonhomogenous Dirichlet boundary conditions. For the error between the exact solution $u(t)$ and the numerical solution $\big\{u_h^n \big\}$ measured in the $L^2$-norm, we have by the triangle inequality
\begin{equation}
\label{all_error}
\big\|u(\tau_n)-u_h^n\big\|_0 \leq \|u(\tau_n)-u_h(\tau_n)\|_0+\big\|u_h(\tau_n)-u_h^n\big\|_0.
\end{equation}
Thus, we will estimate $\|u(\tau_n)-u_h(\tau_n)\|_0$ and $\big\|u_h(\tau_n)-u_h^n\big\|_0$ respectively.

\begin{assumption}
	\label{assumption3}
	The exact solution $u(t)$ satisfies the following regularity conditions:
	\begin{subequations}
		\begin{align}
		\label{assumption3-1} \sup\limits_{t_0 \leq t \leq t_0+T} \|u(t)\|_{2, \Omega} \lesssim 1, \\
		\label{assumption3-2} \sup\limits_{t_0 \leq t \leq t_0+T} \|u_t(t)\|_{2, \Omega} \lesssim 1,
		\end{align}
		where the hidden constants may depend on $t_0$ and $T$.
	\end{subequations}
\end{assumption}

\begin{assumption}
	\label{assumption2}
	The function $f(t)$ is sufficiently smooth with respect to $t$, i.e., for a positive integer $r$, it holds
	\begin{equation}
	\label{smooth}
	\sum\limits_{\alpha \leq r}\Big|f^{(\alpha)}(t) \Big|\lesssim 1, \quad \forall t \in [t_0, t_0+T].
	\end{equation}
\end{assumption}

First of all, the following result readily comes from Theorem 1.2 in \cite{ThomeeVidar2006}.
\begin{lemma}\label{lemma2}
	Let $u_h(t)$ and $u(t)$ be the solutions of \eqref{eq1-5} and \eqref{eq3-1}. Assume that $u(t)$ fulfills Assumption \ref{assumption3}. Then the error bound becomes
	\begin{equation}
	\label{error_semidiscrete}
	\|u_h(t)-u(t)\|_0 \lesssim \|P_hu_0-u_0\|_0+h^2\Big(\|u_0\|_2+\int_0^t \|u_t\|_2 \dd \sigma \Big), \quad t_0\leq t \leq t_0+T.
	\end{equation}
\end{lemma}

A direct consequence of Lemma \ref{lemma2} is
\begin{equation}\label{semidiscrete_error}
\|u(\tau_n)-u_h(\tau_n)\|_0 \lesssim h^2, \quad \forall n = 0, \cdots, N_T,
\end{equation}
where the hidden constant is independent of $h$. 

\begin{theorem}\label{thm1}
	Suppose that the function $f$ satisfies Assumption \ref{assumption2} and the exact solution $u(t)$ satisfies the regularity condition \eqref{assumption3-1} and \eqref{assumption3-2} in Assumption \ref{assumption3}. Assume that $\{u_h^n\}$ are numerical solution at time $\tau_n$ by applying EIFE-linear method sequentially with $s$ different interpolation nodes $c_1, \cdots, c_s$ in $[0, 1]$. Then the numerical error bound becomes
	\begin{equation}\label{fully_error1}
	\big\|u_h(\tau_n)-u_h^n\big\|_0 \lesssim (\Delta \tau)^s,
	\end{equation}
	where the hidden constant is independent of $h$ and $\Delta \tau$.
\end{theorem}

\begin{proof}

	Since all interpolation nodes $c_1, \cdots, c_s$ are fixed, according to the definition of $b_i(-\Delta \tau_nL_h)$ in \eqref{eq3-12} and \eqref{eq1-8} in Lemma \ref{lemma1}, we can immediately obtain that
	\begin{equation*}
	\begin{split}
	\|b_i(-\Delta \tau_n L_h)\|_0&=\Big\|\int_0^1 e^{-\Delta \tau_n (1-\theta) L_h}\prod\limits_{m=1,m \neq i} \frac{\theta-c_m}{c_i-c_m} \dd \theta\Big\|_0\\
	&\leq \int_0^1 \Big\|e^{-\Delta \tau_n(1-\theta)L_h}\Big\|_0 \Big|\prod\limits_{m=1,m \neq i} \frac{\theta-c_m}{c_i-c_m}\Big| \dd \theta\\
	&\lesssim \int_0^1  \Big|\prod\limits_{m=1,m \neq i}\frac{\theta-c_m}{c_i-c_m}\Big|\dd \theta \lesssim 1.
	\end{split}
	\end{equation*}
	Therefore, $\|b_i(-\Delta \tau_n L_h)\|_{0}$ are uniformly bounded for $\Delta \tau_n >0, i = 1, \cdots, s$.
	
	Next, define
	\begin{equation*}
	\psi_j(-\Delta \tau_n L_h)=\varphi_j(-\Delta \tau_n L_h)-\sum\limits_{i=1}^s b_i(-\Delta \tau_n L_h)\frac{c_i^{j-1}}{(j-1)!},
	\end{equation*}
	where
	\begin{equation*}
	\varphi_j(-\Delta \tau_n L_h)=\int_{0}^1 e^{-\Delta \tau_n(1-\theta)L_h}\frac{\theta^{j-1}}{(j-1)!} \dd \theta, \quad j \geq 1.
	\end{equation*}
	Then according to the Lagrange interpolation formula, we find that, for $1 \leq j \leq s$,
	\begin{equation*}
	\int_0^1 e^{-(1-\theta)\Delta \tau_n L_h}\frac{\theta^{j-1}}{(j-1)!}\dd \theta = \int_0^1 \frac{e^{-\Delta \tau_n(1-\theta)L_h}}{(j-1)!}\sum\limits_{i=1}^s \prod_{m=1,m\neq i}^s\frac{\theta-c_m}{c_i - c_m}c_i^{j-1}\dd \theta.
	\end{equation*}
	
	Therefore,
	\begin{equation*}
	\varphi_j(-\Delta \tau_n L_h) = \sum\limits_{i=1}^s b_i(-\Delta \tau_n L_h) \frac{c_i^{j-1}}{(j-1)!},
	\end{equation*}
	which implies that
	\begin{equation*}
	\psi_j(-\Delta \tau_n L_h) = 0, \quad 1 \leq j \leq s.
	\end{equation*}
	
	Since $f$ fulfills Assumption \ref{assumption2} and $P_h$ is a stable projection operator with respect to $L^2$-norm, we can derive that $(P_hf)^{(s)} \in L^1(\Omega)$ also holds. Hence, according to the preceding arguments and Lemma \ref{lemma1}, we know the assumptions of Theorem 2.7 in \cite{HochbruckOstermann2010} are held,  so we can obtain
	\begin{equation*}
	\|u_h(\tau_n) - u_h^n\|_{0} \lesssim \Delta \tau^s
	\end{equation*}
	in view of this theorem.
	
\end{proof}

\begin{remark}
	It's worth noting that no restriction on the time step size $\Delta \tau$ is imposed in Theorem \ref{thm1}, which implies that the proposed EIFE-linear method is stable with large time stepping and appropriate to be applied to coarse grids. Moreover, the temporal accuracy of EIFE-linear method can reach to any order when the number of interpolation nodes is sufficiently large.
\end{remark}

\section{Convergence analysis of PEIFE-linear method for linear problems}\label{PFEM_theory}

Since PEIFE-linear method can be regarded as a modified EIFE-linear method with significant speedup, it's crucial to show that computing PEIFE-linear method in distribution can be as accurate as computing EIFE-linear method sequentially. In this section, we turn to discuss the error analysis of PEIFE-linear method for a linear parabolic problem with the homogenous Dirichlet boundary condition, i.e., the problem \eqref{eq3-11}. The conclusion can be extended to the equations with nonhomogenous Dirichlet boundary conditions with minor revision. The corresponding notations are same as that defined in Section \ref{algorithm-description}.

In the following analysis, we always denote $v(t)(\bm x)=v(t,\bm x)$ and $w(t)(\bm x)=w(t,\bm x)$ for $\bm x \in \Omega$ for brevity. Also, we assume the time intervals are uniform, i.e., $\Delta t = \Delta t_0 = \cdots = \Delta t_{NM-1}=\frac{T}{NM}$ and $\Delta T=\Delta T_0=\cdots=\Delta T_{N-1}=\frac T N$. We also represent $u_h(t)$ as the solution of the semi-discrete (in space) problem \eqref{eq1-5} (or \eqref{eq1-6}), and $\big\{U_{h}^{n,(k)} \big\} $ as the numerical solution produced by the PEIFE-linear method (see Algorithm \ref{algorithm1}). Following ideas about analyzing Parareal method in \cite{GanderVandewalle2007, Guillaume2007}, we can first suppose the chosen coarse propagator $\mathcal{G}$ and fine propagator $\mathcal{F}$ satisfy the following assumptions. 

\begin{assumption}
	Suppose $V_h$ is the finite element space with norm $\|\cdot \|_{0}$, and
	
	\label{assumption1}
	\begin{enumerate}[(1)]
		\item Fine propagator $\mathcal{F}$ fulfills the following regularity condition:
		\begin{equation}
		\label{assumption_eq1}
		\begin{split}
		\max\limits_{0 \leq n \leq N-1}\|\mathcal{F}^M(v;T_n,T_{n+1})-\mathcal{H}(v;T_n,T_{n+1})\|_{0} &\lesssim (\Delta t)^{q+1},
		\end{split}
		\end{equation}
		where $\Delta t = \max\{\Delta t_{l}|l=0, \cdots, NM-1  \}$ and the hidden constant is positive and independent of $\Delta t$ and $v$.
		\item Coarse propagator $\mathcal{G}$ fulfills the following regularity condition:
		\begin{equation}
		\label{assumption_eq2}
		\begin{split}
		\max\limits_{0 \leq n \leq N-1}\|\mathcal{G}(v;T_n,T_{n+1})-\mathcal{H}(v;T_n,T_{n+1})\|_{0} &\lesssim (\Delta T)^{p+1},
		\end{split}
		\end{equation}
		where $\Delta T = \max\{\Delta T_n|n=0, \cdots, N - 1\}$ and the hidden constant is positive and independent of $\Delta T$ and $v$.
		\item Coarse propagator $\mathcal{G}$ can satisfy the following Lipschitz conditions:
		\begin{align}
		\label{assumption_eq3}
		&\|\mathcal{G}(v;T_n,T_{n+1})-\mathcal{G}(w;T_n,T_{n+1})\|_{0} - \|v-w\|_0 \lesssim \Delta T\|v-w\|_{0},\\
		\label{assumption_eq4}
		&\|\mathcal{G}(v;T_n,T_{n+1})-\mathcal{G}(w;T_n,T_{n+1})+\mathcal{H}(w;T_n,T_{n+1})-\mathcal{H}(\xi;T_n,T_{n+1})\|_0\nonumber\\
		& -\|v-\xi\|_0 \lesssim \Delta T\|v-\xi\|_0+(\Delta T)^{p+1}\|w-\xi\|_0,
		\end{align}
		
		where the hidden constants are positive and independent of $\Delta T, v, w$ and $\xi$.
	\end{enumerate}
\end{assumption}

We present below a lemma related to the semigroup $\{e^{-t L_h} \}_{t\geq 0}$ with the infinitesimal generator $(-L_h)$ given by \eqref{A-h}, which has been illustrated in Lemma 3.1 in \cite{HuangJu2022} and is important to the forthcoming numerical analysis.

\begin{lemma}\label{lemma1}
	Let $V_h$ be the finite element space given in \eqref{tensor-FES}, which is equipped with the norm $\|\cdot\|_0$. Then there holds
	\begin{equation}
	\label{eq1-8}
	\|e^{-\tau L_h}\|_0 + \|\tau^{\gamma}L_h^{\gamma}e^{-\tau L_h}\|_0 \lesssim 1, \quad \forall \tau \geq 0, h > 0,
	\end{equation}
	where $\gamma>0$ is any given parameter, and the hidden constant is positive and independent of $h, \tau$ but may depend on $\gamma$.
\end{lemma}

\begin{theorem}
	\label{thm3}
	Suppose that $f(t)$ fulfills Assumption \ref{assumption2} with some positive integer $q$. Assume that EIFE-linear method is chosen to fine grids with $q$ different interpolation nodes $c_1, \cdots, c_q \in [0, 1]$. If $v(t)$ is sufficiently smooth with respect to $t$, then 
	\begin{equation*}
	\begin{split}
	\max\limits_{0 \leq n \leq N-1}\|\mathcal{F}^M(v;T_n,T_{n+1})-\mathcal{H}(v;T_n,T_{n+1})\|_{0} &\lesssim \Delta T(\Delta t)^{q},
	\end{split}
	\end{equation*}
	is fulfilled and the hidden constant is independent of $\Delta T, \Delta t, v$, which means that EIFE-linear method is proper to be applied to fine grids.
\end{theorem}
\begin{proof}
	For any $1\leq n \leq N$, based on the Parareal-EIFE method proposed in Algorithm \ref{algorithm1}, we can derive by recurrence that
	\begin{align}
	\label{thm3-1}
	&\mathcal{F}^M(v;T_n,T_{n+1})-\mathcal{H}(v;T_n,T_{n+1})\notag\\
	=&\sum\limits_{j=0}^{M-1}e^{-\Delta t(M-1-j)L_h}\Delta t\sum\limits_{i=1}^q b_i(-\Delta tL_h)P_hf(T_n+(j+c_i)\Delta t)-\int_0^{\Delta T}e^{-(\Delta T-\theta)L_h}P_hf(T_n+\theta)\dd \theta.
	\end{align}
	
	Concerning about the regularity condition, since $t_{nM}=T_n$ and the definition of $b_i(-\tau_nL_h)$ in \eqref{eq3-12}, according to the Lagrangian interpolation theorem, we can derive
	\begin{align}
	\label{thm3-2}
	&\|\mathcal{F}^M(v;T_n,T_{n+1})-\mathcal{H}(v;T_n,T_{n+1})\|_0\notag\\
	=&\Big\|\sum\limits_{j=0}^{M-1}e^{-\Delta t(M-1-j)L_h}\Delta t\sum\limits_{i=1}^qP_hf(T_n+(j+c_i)\Delta t)\int_0^1 e^{-\Delta t(1-\theta )L_h}l_i(\theta)\dd \theta \notag\\
	&-\Delta t\sum\limits_{j=0}^{M-1}e^{-(M-j-1)\Delta tL_h}\int_0^1 e^{-\Delta t(1-\eta)L_h}P_hf(T_n+(j+\eta )\Delta t)\dd \eta\Big\|\notag\\
	\leq & \Delta t\sup\limits_{0 \leq j\leq M-1}\|e^{-(M-j-1)\Delta tL_h}\|_0\cdot \sup\limits_{0 \leq \eta\leq 1}\|e^{-\Delta t(1-\eta)L_h}\|_0 \sum\limits_{j=0}^{M-1}\int_0^{\Delta t}\frac{M_q}{q!}w_q(\theta)\dd \theta,
	\end{align}
	where $M_q = \sup\limits_{t_0 \leq t \leq t_0+T} |P_hf^{(q)}(t)|$ and $w_q(\theta)=\Big|\prod\limits_{i=1}^q (\theta-c_i)\Big|$ is the Lagrange polynomial of order $q$. According to \eqref{eq1-8} in Lemma \ref{lemma1}, by direct manipulation, we can easily reformulate \eqref{thm3-2} to
	\begin{equation*}
	\|\mathcal{F}^M(v;T_n,T_{n+1})-\mathcal{H}(v;T_n,T_{n+1})\|_0 \lesssim \Delta T (\Delta t)^q.
	\end{equation*}

	In conclusion, EIFE-linear method can fulfill the corresponding regularity condition, then we can know that EIFE-linear is proper to be applied to fine grids.
	
\end{proof}

\begin{theorem}
	\label{thm4}
	Assume that $f(t)$ fulfills Assumption \ref{assumption2} for some positive integer $p$ for linear problem. Suppose that EIFE method is unconditionally stable and chosen to be applied to coarse grids with $p$ different interpolation nodes $c_1, \cdots, c_p \in [0, 1]$. If $v(t)$, $w(t)$, $\xi(t)$ are sufficiently smooth with respect to $t$, then 
	\begin{equation*}
	\begin{split}
	&\max\limits_{0 \leq n \leq N-1}\|\mathcal{G}(v;T_n,T_{n+1})-\mathcal{H}(v;T_n,T_{n+1})\|_{0} \lesssim (\Delta T)^{p+1},\\
	&\|\mathcal{G}(v;T_n,T_{n+1})-\mathcal{G}(w;T_n,T_{n+1})\|_{0} - \|v-w\|_0 \lesssim \Delta T\|v-w\|_{0},\\
	&\|\mathcal{G}(v;T_n,T_{n+1})-\mathcal{G}(w;T_n,T_{n+1})+\mathcal{H}(w;T_n,T_{n+1})-\mathcal{H}(\xi;T_n,T_{n+1})\|_0\\
	&-\|v-\xi\|_0 \lesssim \Delta T\|v-\xi\|_0 + (\Delta T)^{p+1}\|w-\xi\|_0
	\end{split}
	\end{equation*}
	are fulfilled and the hidden constants are all independent of $\Delta T, v, w, \xi$, which means that EIFE-linear method is proper to be applied to coarse grids.
\end{theorem}

\begin{proof}
	Following the similar arguments for deriving Theorem \ref{thm3}, we can easily derive that the coarse propagator $\mathcal{G}$ can fulfill the regularity condition when fixing $M=1$. So the corresponding details will be omitted.
	
	Considering about the Lipschitz continuous condition \eqref{assumption_eq3}, by using \eqref{eq1-8} in Lemma \ref{lemma1},
	\begin{align}
	\label{eq3-4}
	&\|\mathcal{G}(v;T_n,T_{n+1})-\mathcal{G}(w;T_n,T_{n+1})\|_0-\|v-w\|_0\nonumber\\
	=&\Big\|e^{-\Delta TL_h}v+\Delta T\sum\limits_{i=1}^p b_i(-\Delta TL_h)P_hf(T_n+c_i\Delta T)-e^{-\Delta TL_h}w\nonumber\\
	&-\Delta T\sum\limits_{i=1}^p b_i(-\Delta TL_h)P_hf(T_n+c_i\Delta T)\Big\|_0-\|v-w\|_0\nonumber\\
	=& (\|e^{-\Delta TL_h}\|_0-1)\|v-w\|_0.
	\end{align}
	
	As for the estimation of $\|e^{-\Delta TL_h}\|_0-1$, since $L_h$ is symmetric and positive definite, according to Taylor expansion,
	\begin{equation*}
	\begin{split}
	\|e^{-\Delta T_nL_h}\|_0-1 &= \sup\limits_{\lambda \in \lambda(L_h)} e^{-\Delta T\lambda}-1\\
	&\leq e^{-\Delta T\alpha}-1 \lesssim \Delta T.
	\end{split}
	\end{equation*}
	
	Therefore, the estimation of \eqref{eq3-4} becomes
	\begin{equation}
	\label{eq3-6}
	\|\mathcal{G}(v;T_n,T_{n+1})-\mathcal{G}(w;T_n,T_{n+1})\|_0-\|v-w\|_0 \lesssim \Delta T\|v-w\|_0,
	\end{equation}
	which implies that EIFE-linear method fulfills the Lipchitz continuous condition \eqref{assumption_eq3}.
	
	In the same way, the Lipschitz continuous condition \eqref{assumption_eq4} can also be easily derived. Therefore, we can conclude that EIFE-linear is proper to be applied to coarse grids.
	
\end{proof}

\begin{remark}
	When choosing EIFE-linear scheme to both coarse grids and fine grids, it's essential to balance the computational cost and numerical accuracy simultaneously. Generally speaking, we always select a less precise EIFE scheme as the coarse propagator $\mathcal{G}$ while a more precise one as the fine propagator $\mathcal{F}$, which implies that the value of $p$ is usually less than $q$.
\end{remark}

\begin{theorem}
	\label{thm5}
	Suppose that conditions in Assumption \ref{assumption1} are all fulfilled. Then the error bound of Parareal algorithm can reach to the accuracy of applying $\mathcal{F}$ to the fine grids, which implies that the numerical solution $\big\{U_{h}^{n,(k)}\big\}$ produced by (\ref{Parareal}) satisfies
	\begin{equation*}
	\big\|U_{h}^{n,(k)} - u_h(T_n)\big\|_{0} \lesssim (\Delta t) ^{q}
	\end{equation*}
	when $k \geq N$ and the hidden constant is independent of $h$, $\Delta t$ and $\Delta T$.
\end{theorem}

\begin{proof}
	
	Step 1: According to format of Parareal algorithm (see Algorithm \ref{algorithm1}), on the one hand, when $k \geq 1$,
	\begin{equation*}
	\label{the_error}
	\begin{split}
	U_{h}^{n,(k)}-u_h(T_n)
	&=\mathcal{G}(U_{h}^{n-1,(k)};T_{n-1},T_n)+\mathcal{F}^M(U_{h}^{n-1,(k-1)};T_{n-1},T_n)
	-\mathcal{G}(U_{h}^{n-1,(k-1)};T_{n-1},T_n)\\
	&-\mathcal{H}(u_h(T_{n-1});T_{n-1},T_n)=:  {\rm I}_1 + {\rm I}_2,
	\end{split}
	\end{equation*}
	where
	\begin{equation*}
	\begin{split}
	{\rm I}_1 &:= \mathcal{F}^M(U_{h}^{n-1,(k-1)};T_{n-1},T_n)-\mathcal{H}(U_{h}^{n-1, (k-1)};T_{n-1},T_n),\\
	{\rm I}_2 &:= \mathcal{G}(U_{h}^{n-1,(k)};T_{n-1},T_n) -\mathcal{G}(U_{h}^{n-1,(k-1)};T_{n-1},T_n) + \mathcal{H}(U_{h}^{n-1,(k-1)};T_{n-1},T_{n}) - \mathcal{H}(u_h(T_{n-1});T_{n-1},T_n).
	\end{split}
	\end{equation*}
	
	Then owing to the conditions (\ref{assumption_eq1}) and (\ref{assumption_eq4}) are fulfilled, we can derive that
	\begin{equation}
	\label{error1}
	\begin{split}
	&\big\|U_{h}^{n,(k)}-u_h(T_n)\big\|_{0} - \big\|U_{h}^{n-1,(k)}-u_h(T_{n-1})\big\|_0\\
	\lesssim & (\Delta T)^{p+1}\big\|U_{h}^{n-1,(k-1)}-u_h(T_{n-1})\big\|_0+\Delta T\big\|U_{h}^{n-1,(k)}-u_h(T_{n-1})\big\|_0+\Delta T(\Delta t)^{q},\quad k \geq 1.
	\end{split}
	\end{equation}
	
	On the other hand, when $k=0$, then
	\begin{equation*}
	\begin{split}
	U_{h}^{n,(0)}-u_h(T_n)=& \mathcal{G}(U_{h}^{n-1,(0)};T_{n-1},T_n)-\mathcal{H}(u_h(T_{n-1});T_{n-1},T_n)\\
	=& \Big(\mathcal{G}(U_{h}^{n-1,(0)};T_{n-1},T_n)-\mathcal{G}(u_h(T_{n-1});T_{n-1},T_n)\Big) \\
	&+ \Big(\mathcal{G}(u_h(T_{n-1});T_{n-1},T_n)-\mathcal{H}(u_h(T_{n-1});T_{n-1},T_n)\Big).
	\end{split}
	\end{equation*}
	
	Therefore, noting the conditions \eqref{assumption_eq2} and \eqref{assumption_eq3}, we can derive that
	\begin{equation}
	\label{k=0}
	\begin{split}
	\big\|U_{h}^{n,(0)}-u_h(T_n)\big\|_{0}-\big\|U_{h}^{n-1,(0)}-u_h(T_{n-1})\big\|_{0}
	\lesssim \Delta T\big\|U_{h}^{n-1,(0)}-u_h(T_{n-1})\big\|_{0}+(\Delta T)^{p+1}.
	\end{split}
	\end{equation}
	
	Define 
	\begin{equation*}
	\begin{split}
	\alpha-1 \eqsim \Delta T, \quad \beta \eqsim (\Delta T)^{p+1}, \quad \gamma \eqsim \Delta T(\Delta t)^{q},\quad \delta \eqsim (\Delta T)^{p+1},
	\end{split}
	\end{equation*}
	where the hidden constants are positive and independent of $\Delta t, \Delta T$. Denote $\tilde{e}_n^{(k)}=\big\|U_{h}^{n,(k)}-u_h(T_n)\big\|_{0}$, then (\ref{error1}) and (\ref{k=0}) can be reformulated as:
	\begin{equation}
	\label{error2}
	\left\{
	\begin{split}
	\tilde{e}_n^{(k)}&\leq \alpha \tilde{e}_{n-1}^{(k)}+\beta \tilde{e}_{n-1}^{(k-1)}+\gamma, \\
	\tilde{e}_n^{(0)}&\leq \alpha \tilde{e}_{n-1}^{(0)} +\delta.
	\end{split}
	\right.
	\end{equation}

	Step 2: Let $e_n^{(k)}$ as intermediate variables for comparison, which satisfy
	\begin{equation}
	\label{compare_eq}
	\left\{\begin{split}
	e_n^{(k)}&=\alpha e_{n-1}^{(k)}+\beta e_{n-1}^{(k-1)}+\gamma, \\
	e_n^{(0)}&=\alpha e_{n-1}^{(0)}+\delta.
	\end{split} \right.
	\end{equation}
	
	Define $\rho^{(k)}(\tau)=\sum\limits_{n=1}^{\infty}e_n^{(k)}\tau^n$ with $\tau$ sufficiently small. Since $e_0^{(k)}=\tilde{e}_0^{(k)}=0$, (\ref{compare_eq}) can be converted to
	\begin{equation*}
	\label{compare_eq3}
	\left\{
	\begin{split}
	\rho^{(k)}(\tau) &= \alpha \tau \rho^{(k)}(\tau) +\beta \tau \rho^{(k-1)}(\tau)+ \frac{\gamma\tau}{1-\tau},\\
	\rho^{(0)}(\tau) &= \alpha \tau \rho^{(0)}(\tau) + \frac{\delta \tau}{1-\tau}.
	\end{split}
	\right.
	\end{equation*}
	
	After simple manipulations, we can obtain that
	\begin{equation*}
	\left\{
	\begin{split}
	&\rho^{(0)}(\tau) = \frac{\delta \tau}{(1-\alpha \tau)(1-\tau)},\\
	&\rho^{(k)}(\tau) = \Big(\frac{\beta \tau}{1-\alpha \tau}\Big)^k\rho^{(0)}(\tau)+ \Big(1-\Big(\frac{\beta \tau}{1-\alpha \tau}\Big)^k\Big)\frac{\gamma\tau}{(1-\tau)(1-(\alpha+\beta)\tau)}.
	\end{split}
	\right.
	\end{equation*}
	
	Since $\alpha>1$, $\alpha+\beta>1$ and $\tau > 0$ is sufficiently small, after direct manipulation, we can derive that
	\begin{subequations}
		\begin{align}
		\label{mid2}
		\frac{\gamma \tau}{(1-\tau)(1-(\alpha+\beta)\tau)}\Big(1-\Big( \frac{\beta\tau}{1-\alpha\tau}\Big)^k\Big)&\leq \frac{\gamma\tau}{(1-(\alpha+\beta)\tau)^2},\\
		\label{mid3}
		\Big(\frac{\beta \tau}{1-\alpha \tau} \Big)^k\rho^{(0)}(\tau) & \leq \frac{\beta^k\delta \tau^{k+1}}{(1-\alpha\tau)^{k+2}}.
		\end{align}
	\end{subequations}	
	
	Following the equality \eqref{mid2}-\eqref{mid3}, we can derive that	
	\begin{equation}
	\label{scaling}
	\begin{split}
	\rho^{(k)}(\tau) \leq \frac{\beta^k \delta \tau^{k+1}}{(1-\alpha \tau)^{k+2}}+\frac{\gamma\tau }{(1-(\alpha+\beta)\tau)^2}.
	\end{split}
	\end{equation}
	
	According to Taylor expansion, then (\ref{scaling}) is transformed as
	\begin{equation*}
	\label{Taylor}
	\begin{split}
	\rho^{(k)}(\tau) &= \sum\limits_{n=1}^{\infty}e_n^{(k)}\tau^n \leq \beta^k\delta\tau^{k+1}\sum\limits_{j=0}^{\infty}C_{k+1+j}^{j}\alpha^j \tau^j+\gamma \sum\limits_{j=0}^{\infty}C_{1+j}^j(\alpha+\beta)^j\tau^{j+1}.
	\end{split}
	\end{equation*}
	
	Therefore, if $k \geq n$, we can obtain that
	\[e_n^{(k)}\leq C_{n}^{n-1}(\alpha+\beta)^{n-1}\gamma. \]
	It's obvious that $\tilde{e}_n^{(k)}\leq e_n^{(k)}$, thus $\tilde{e}_n^{(k)}\leq C_{n}^{n-1}(\alpha+\beta)^{n-1}\gamma$.
	
	Especially, if choosing $n = N$, which is $k \geq N$, then
	\begin{equation*}
	\label{k_greater_n2}
	\begin{split}
	\tilde{e}_n^{(k)} &\leq C_{n}^{n-1}(\alpha+\beta)^{n-1}\gamma\\
	&\lesssim n(1+C\Delta T)^{n-1}\Delta T(\Delta t)^{q}\\
	&\lesssim (\Delta t)^{q}.
	\end{split}
	\end{equation*}
	Therefore, when $N$ is fixed, the error bound of coarse grids can achieve the accuracy of fine grids in finite steps.
	
\end{proof}

\begin{remark}
	In reality, for solving large-scale problems, Parareal method can achieve great accuracy when the iteration times is much smaller than the number of processors, i.e. $k \ll N$. Moreover, the number of processors can not be too large, since the cost of communication and sequential correction will increase with the number of processors.
\end{remark}

\begin{theorem}\label{thm2}
	Suppose that the function $f$ satisfies Assumption \ref{assumption2} and the exact solution $u(t)$ satisfies the regularity condition \eqref{assumption3-1} and \eqref{assumption3-2} in Assumption \ref{assumption3}. Assume that EIFE-linear method is chosen to both coarse grids and fine grids with $p$ and $q$ different interpolation nodes in $[0, 1]$ respectively ($p \leq q, q \geq 1$). Then the numerical solution $\{U_{h}^{n,(k)}\}$ produced by PEIFE-linear method satisfies
	\begin{equation}\label{fully_error}
	\big\|u(T_n)-U_{h}^{n,(k)}\big\|_0 \lesssim (\Delta t)^q + h^2,
	\end{equation}
	for $k \geq N$, and the hidden constant is independent of $h$ and $\Delta t$.
\end{theorem}

With the help of triangle inequality, the above conclusion can be immediately obtained by combining the results of Theorem \ref{thm1} and Theorem \ref{thm5}.

\begin{remark}
	We can reformulate the model problem \eqref{eq1-1} with periodic boundary condition by transforming $v(t,x)=e^{-\alpha t}u(t,x), \alpha >0$, which leading the bilinear form \eqref{bilinear} coercive. In addition, the spatial mesh can also be relaxed to quasi-uniform ones while similar theoretical analysis still remains valid.
\end{remark}

\section{Numerical experiments and applications}\label{numerical}

In this section, several numerical examples and applications are presented to illustrate the performance of the proposed efficient PEIFE-linear method and verify the convergent order we have derived in Section \ref{linear_theory}. All numerical experiments are done using Matlab on an Intel i5-8250U, 1.80GHz CPU laptop with 8GB memory.

\subsection{Convergence tests}
We verify the error estimates obtained in Theorem \ref{thm2} for PEIFE-linear scheme, and the numerical errors $\big\|U_{h}^{n,(k)}-u(T_n)\big\|_0$ are evaluated at the terminal time $t_0+T$. Specially, we denote PEIFE-linear22 scheme, PEIFE-linear23 scheme and PEIFE-linear33 scheme as PEIFE-linear method \eqref{Parareal} with (p, q) = (2, 2), (2, 3) and (3, 3), respectively. Also, we denote EIFE-linear2 and EIFE-linear3 as EIFE-linear method \eqref{eq1-7} with 2 and 3 Runge-Kutta stages. For brevity, we denote the time steps $N_T$ as the result of $N\times M$ for all PEIFE-linear methods and also denote $\big\|U_{h}^{n,(k)}-u(T_n)\big\|_0$ as the numerical error for all EIFE-linear methods. All interpolation nodes are selected uniformly in [0, 1].

\begin{example}
	\label{ex5}
	In this example, we consider the following one-dimensional linear reaction-diffusion problem with homogeneous Dirichlet boudnary condition:
	\begin{equation*}
	\left\{
	\begin{split}
	&u_t = u_{xx}+(2+x(1-x))e^t, \quad x \in [0, 1], \quad 0 \leq t \leq 1,\\
	&u(0, x) = x(1-x), \quad x \in [0, 1], \\
	\end{split}
	\right.
	\end{equation*}
	The exact solution is given by $u(t, x) = x(1-x)e^t$.
\end{example}

On the one hand, for the spatial accuracy tests, we run the PEIFE-linear33 scheme with fixed time steps $N_T=N\times M=4\times 256$ (i.e., $\Delta T=T/N=1/4,\Delta t=\Delta T/M=1/1024$) and uniformly refined spatial meshes with $N_x=8,16,32$ and 64, respectively, so that the spatial mesh sizes are much coarser compared to the time step size. On the other hand, for the temporal accuracy tests, we fix the spatial meshes with $N_x=4096$ and run the EIFE-linear2, PEIFE-linear22 with uniform time step $N_T=N\times M=4\times 2, 4\times 4, 4\times 8, 4\times 16$ and EIFE-linear3, PEIFE-linear22, PEIFE-linear23 with $N_T=N\times M=4 \times 1, 4 \times 2, 4 \times 4, 4\times 8$, respectively. In order to verify the temporal accuracy of PEIFE-linear method can reach to that of EIFE-linear method when $k \geq N$, we always show the numerical results of PEIFE-linear method after iterating for 4 times. All numerical results are reported in Table \ref{tab 3-1}, including the solution errors at the terminal time measured in the $L^2$-norm and corresponding convergence rates. We can observe the second-order spatial convergence with respect to $L^2$-norm for all numerical schemes. It's also easy to find the second-order temporal convergence for the EIFE-linear2 and PEIFE-linear22 schemes, and third-order temporal convergence for the EIFE-linear3, PEIFE-linear23 and PEIFE-linear33 schemes in $L^2$-norm. What's more, the numerical accuracy of all PEIFE-linear methods depends on the corresponding fine propagator and can achieve the accuracy of applying EIFE-linear method sequentially with fine time step, which coincide very well with the error estimates derived in Theorem \ref{thm2}.

\begin{table}[ht]
	\centering
	\caption{Numerical results on the solution errors measured in the $L^2$-norm and corresponding convergence rates for the EIFE-linear2, EIFE-linear3, PEIFE-linear22, PEIFE-linear33 and PEIFE-linear23 schemes in Example \ref{ex5}}
	\begin{tabular}{|ccccc|}
		\hline
		Method & $N_T$ & $N_x$ & $\big\|U_{h}^{n,(4)}-u(t_n)\big\|_0$ & CR\\
		\hline
		\multicolumn{5}{|c|}{Spatial accuracy tests}\\
		\hline
		\multirow{4}{40pt}{PEIFE-linear33} & $4 \times 256$ & 8 & 7.3000e-03 & -\\
		& $4 \times 256$ & 16 & 1.8000e-03 & 2.02\\
		& $4 \times 256$ & 32 & 4.5453e-04 & 1.99\\
		& $4 \times 256$ & 64 & 1.1364e-04 & 2.00\\
		\hline
		\multicolumn{5}{|c|}{Temporal accuracy tests}\\
		\hline
		\multirow{4}{40pt}{EIFE-linear2} & 8 & 4096 & 5.2732e-04 & -\\
		& 16 & 4096 & 1.0660e-04 & 2.31\\
		& 32 & 4096 & 2.3429e-05 & 2.19\\
		& 64 & 4096 & 5.4701e-06 & 2.10\\
		\hline
		\multirow{4}{40pt}{EIFE-linear3} & 4 & 4096 & 7.2503e-05 & -\\
		& 8 & 4096 & 6.7772e-06 & 3.42\\
		& 16& 4096 & 7.2130e-07 & 3.23\\
		& 32& 4096 & 1.0250e-07 & 2.81\\
		\hline
		\multirow{4}{40pt}{PEIFE-linear22} & $4 \times 2$ & 4096 & 5.2732e-04 & -\\
		& $4\times 4$ & 4096 & 1.0660e-04 & 2.31\\
		& $4\times 8$ & 4096 & 2.3429e-05 & 2.19\\
		& $4\times 16$ & 4096 & 5.4701e-06 & 2.10\\
		\hline
		\multirow{4}{40pt}{PEIFE-linear23} & $4\times 1$ & 4096 & 7.2503e-05 & -\\
		& $4\times 2$ & 4096 & 6.7772e-06 & 3.42\\
		& $4\times 4$ & 4096 & 7.2130e-07 & 3.23\\
		& $4\times 8$ & 4096 & 1.0250e-07 & 2.81\\
		\hline
		\multirow{4}{40pt}{PEIFE-linear33} & $4\times 1$ & 4096 & 7.2503e-05 & -\\
		& $4\times 2$ & 4096 & 6.7772e-06 & 3.42\\
		& $4\times 4$ & 4096 & 7.2130e-07 & 3.23\\
		& $4\times 8$ & 4096 & 1.0250e-07 & 2.81\\		
		\hline
	\end{tabular}
	\label{tab 3-1}
\end{table}

\begin{example}\label{ex6}
	In this example, we consider the following two-dimension linear reaction-diffusion problem with homogeneous Dirichlet boundary condition, which is also studied in \cite{JuZhang2015}.
	\begin{equation*}
	\left\{ 
	\begin{split}
	&u_t = \Delta u + \pi^2 e^{-4\pi^2 t}\sin(\pi(x-\frac 1 4))\sin(2\pi(y-\frac 1 8)), \quad (x, y)\in \Omega\quad 0 \leq t \leq T,\\
	&u(0, x, y) = \sin(\pi(x-\frac 1 4))\sin(2\pi(y-\frac 1 8)), \quad (x, y) \in \Omega,\\
	\end{split}
	\right.
	\end{equation*}
	where $\Omega = [\frac 1 4, \frac 5 4] \times [\frac 1 8, \frac 5 8]$ and the terminal time $T=0.6$. The exact solution is given by $u(t, x, y) = e^{-4\pi^2 t}\sin(\pi(x-\frac 1 4))\sin(2\pi(y-\frac 1 8))$.
\end{example}

For the spatial accuracy tests, we run the PEIFE-linear33 scheme with fixed $N_T=N\times M=4\times 256$ and uniformly refined spatial meshes with $N_x \times N_y=8\times 4, 16 \times 8, 32 \times 16$ and $64 \times 32$, respectively, so that the spatial mesh sizes are much coarser than time step size. For the temporal accuracy tests, we fix the spatial meshes with $N_x \times N_y=2048\times 1024$ and run the EIFE-linear2, PEIFE-linear22 with uniform time step $N_T=N\times M=4\times 4, 4\times 8, 4 \times 16, 4\times 32$ and EIFE-linear3, PEIFE-linear23, PEIFE-linear33 with $N_T=N\times M=4\times 2, 4\times 4, 4\times 8, 4\times 16$, respectively. In order to verify the temporal accuracy of PEIFE-linear method can reach to that of EIFE-linear method when $k \geq N$, we always show the numerical results of PEIFE-linear method after iterating for 4 times. All numerical results are reported in Table \ref{tab 4-1}, including the solution errors at the terminal time measured in the $L^2$-norm and corresponding convergence rates. We still observe the roughly second-order spatial convergence in the $L^2$-norm. It is also seen that the temporal convergence is slightly higher than two for EIFE-linear2 and PEIFE-linear22 schemes and slightly higher than three for EIFE-linear3, PEIFE-linear23 and PEIFE-linear33 schemes, which basically matches the results for Theorem \ref{thm2}.

\begin{table}[htbp]
	\centering
	\caption{Numerical results on the solution errors measured in the $L^2$-norm and corresponding convergence rates for the EIFE-linear2, EIFE-linear3, PEIFE-linear22, PEIFE-linear23 and PEIFE-linear33 schemes in Example \ref{ex6}}
	\begin{tabular}{|ccccc|}
		\hline
		Method & $N_T$ & $N_x \times N_y$ & $\big\|U_{h}^{n,(4)}-u(t_n)\big\|_0$ & CR\\
		\hline
		\multicolumn{5}{|c|}{Spatial accuracy tests}\\
		\hline
		\multirow{4}{40pt}{PEIFE-linear33}& $4 \times 256$ & $8 \times 4$ & 3.3512e-12 & -\\
		& $4 \times 256$ & $16 \times 8$ & 9.5480e-13 & 1.81\\
		& $4 \times 256$ & $32 \times 16$ & 2.4738e-13 & 1.95\\
		& $4 \times 256$ & $64 \times 32$ & 6.2516e-14 & 1.98\\
		\hline
		\multicolumn{5}{|c|}{Temporal accuracy tests}\\
		\hline
		\multirow{4}{40pt}{EIFE-linear2}& 16& $2048 \times 1024$ & 4.0690e-12 & -\\
		&32 & $2048 \times 1024$ & 6.5336e-13 & 2.64\\
		&64 & $2048\times 1024$ & 1.3075e-13 & 2.32\\
		&128 & $2048 \times 1024$ & 2.9349e-14 & 2.16\\
		\hline
		\multirow{4}{40pt}{EIFE-linear3} & 8 & $2048 \times 1024$ & 1.1463e-11 & -\\
		& 16 & $2048\times 1024$ & 6.0057e-13 & 4.25\\
		&32 & $2048 \times 1024$ & 4.9785e-14 & 3.59\\
		&64 & $2048\times 1024$ & 4.9332e-15 & 3.36\\
		\hline
		\multirow{4}{40pt}{PEIFE-linear22} & $4\times 4$ & $2048 \times 1024$ & 4.0690e-12 & -\\
		& $4 \times 8$ & $2048 \times 1024$ & 6.5336e-13 & 2.64\\
		& $4 \times 16$ & $2048 \times 1024$ & 1.3075e-13 & 2.32\\
		& $4 \times 32$ & $2048 \times 1024$ & 2.9349e-14 & 2.16\\
		\hline
		\multirow{4}{40pt}{PEIFE-linear23} & $4 \times 2$ & $2048 \times 1024$ & 1.1463e-11 & -\\
		& $4 \times 4$ & $2048 \times 1024$ & 6.0057e-13 & 4.25\\
		& $4 \times 8$ & $2048 \times 1024$ & 4.9785e-14 & 3.59\\
		& $4 \times 16$ & $2048 \times 1024$ & 4.9332e-15 & 3.36\\
		\hline
		\multirow{4}{40pt}{PEIFE-linear33} & $4 \times 2$ & $2048 \times 1024$ & 1.1463e-11 & -\\
		& $4 \times 4$ & $2048 \times 1024$ & 6.0057e-13 & 4.25\\
		& $4 \times 8$ & $2048 \times 1024$ & 4.9785e-14 & 3.59\\
		& $4 \times 16$ & $2048 \times 1024$ & 4.9332e-15 & 3.36\\
		\hline
	\end{tabular}
	\label{tab 4-1}
\end{table}

The overall cost per iteration of PEIFE-linear method is tested with fixed $N\times M=4\times 16,k=4$ and uniform spatial meshes with $N_x \times N_y=256\times 128, 512\times 256, 1024\times 512$ and $2048\times 1024$. Since the analysis of computing cost of PEIFE-linear22, PEIFE-linear23 schemes are same as PEIFE-linear33 scheme, we only test the running time with PEIFE-linear33 scheme on 4 cores. Table \ref{tab4-2} reports the average CPU time costs (seconds) per iteration for the PEIFE-linear33 scheme and correponding growth factors along the refinement of the spatial mesh. The results clearly show that the computational cost grows almost linearly along with the number of mesh nodes, which matches well with the property of FFT and demonstrates the high efficiency of our PEIFE-linear method.

\begin{table}[htbp]
	\centering
	\caption{The average running time costs (seconds) per interation under different spatial meshes and corresponding growth factors with respect to the number of mesh nodes for the PEIFE-linear33 scheme in Example \ref{ex6}}
	\begin{tabular}{|cccc|}
		\hline
		$N_x \times N_y$ & $N \times M$ & Average running time  & Growth factor\\
		& &cost per step& \\
		\hline
		$256 \times 128$ & $4\times 16$ & 22.16 & - \\
		$512 \times 256$ & $4 \times 16$ & 88.65 & 1.00\\
		$1024 \times 512$ & $4 \times 16$ & 336.38 & 0.96\\
		$2048 \times 1024$ & $4 \times 16$ & 1378.28 & 1.02\\
		\hline
	\end{tabular}
	\label{tab4-2}
\end{table}

Furthermore, we also compare the performance of all PEIFE-linear methods and EIFE-linear methods. We compare the running time (seconds) and numerical error in $L^2$ and $L^{\infty}$ norms of various PEIFE-linear and EIFE-linear methods. We run the efficiency tests with fixed $N_x \times N_y=1024\times 512$ and $N_T=N\times M=8\times 8$. As for the testing for PEIFE-linear method, we run the programs in parallel on 8 cores with $k=2$ and the EIFE-linear code sequentially on 1 core. Table \ref{tab4-3} shows again that the accuracy of PEIFE-linear method can reach to that of applying EIFE-linear method sequentially with fine step size, and PEIFE-linear method can achieve significant speedup of roughly 2. Moreover, The $L^2$ and $L^{\infty}$ error declining curves for PEIFE-linear methods are shown in Figure \ref{fig4}. We observe that the error curves decline rapidly and all PEIFE-linear methods achieve convergent within only 1 iteration, which is much smaller than the number of coarse time subintervals (i.e., $N=8$), thus all of them are very efficient due to the fast convergence.

\begin{table}
	\centering
	\caption{The running time (seconds) and numerical error with respect to $L^2$ and $L^{\infty}$ norms for various PEIFE-linear and EIFE-linear methods in Example \ref{ex6}}
	\begin{tabular}{|cccc|}
		\hline
		Method & $L^2$-error & $L^{\infty}$-error & running time\\
		\hline
		EIFE-linear2 & 1.3093e-13 & 3.7007e-13 & 1542.76\\
		EIFE-linear3 & 4.7510e-15 & 1.3680e-14 & 1817.81\\
		PEIFE-linear22 &1.3093e-13 & 3.7007e-13 & 892.90\\
		PEIFE-linear23 & 4.7510e-15 & 1.3680e-14 & 1058.62\\
		PEIFE-linear33 & 4.7510e-15 & 1.3680e-14 & 1107.86\\
		\hline
	\end{tabular}
\label{tab4-3}
\end{table}

\begin{figure}[htbp]
	\centering
	\subfigure{
		\centering
		\includegraphics[width = 170pt,height=150pt]{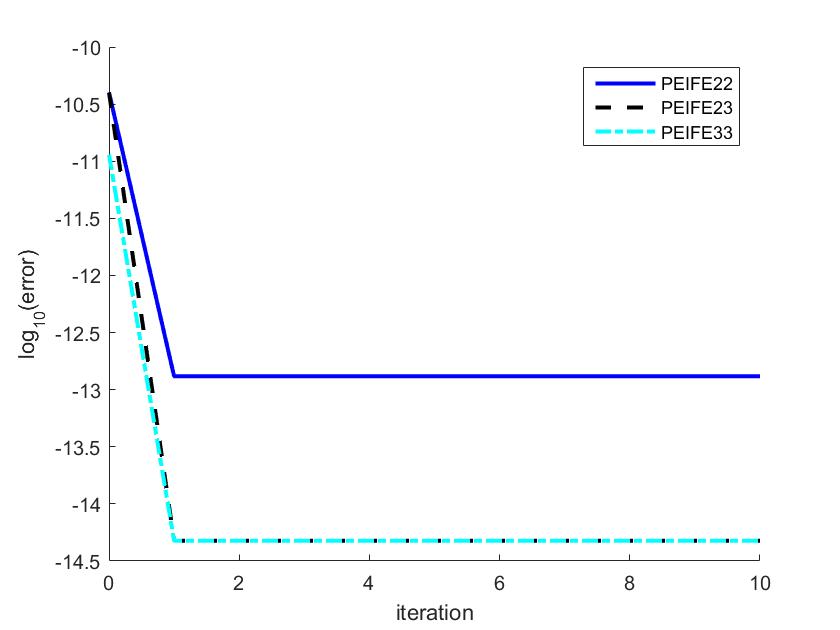}
		\label{fig4-1}
	}
	\subfigure{
		\centering
		\includegraphics[width = 170pt,height=150pt]{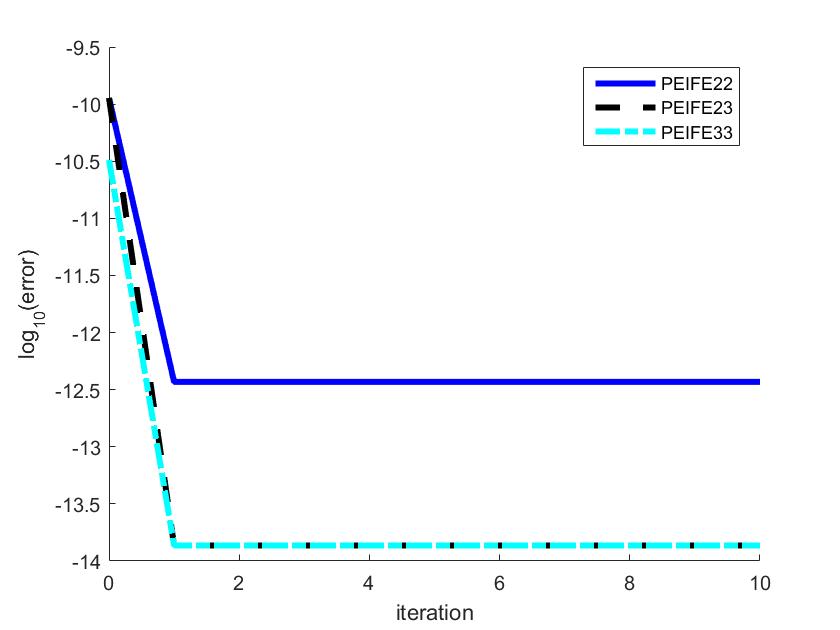}
		\label{fig4-2}
	}
	\caption{The evolutions of $L^2$-error and $L^{\infty}$-error declining curves of the numerical solution produced by PEIFE-linear22, PEIFE-linear23 and PEIFE-linear33 schemes for Example \ref{ex6}.}
	\label{fig4}
\end{figure}

\begin{example}\label{ex1}
In this example, we consider the following three-dimension linear reaction-diffusion problem with homogeneous Dirichlet boundary condition,
\begin{equation*}
\left\{\begin{split}
&u_t= \frac 1 8 \Delta u +2 \pi^2 e^{-4\pi^2 t}\sin(4\pi(x-\frac 1 4))\sin(4\pi (y-\frac 1 8))\sin(4\pi (z-\frac 1 2)),\quad (x, y, z)\in \Omega, \quad 0 \leq t \leq T,\\
&u(0, x, y, z) = \sin(4\pi (x-\frac 1 4))\sin(4\pi (y-\frac 1 8))\sin(4\pi (z-\frac 1 2)), \quad (x,y,z) \in \Omega, \\
\end{split}  \right.
\end{equation*}
where $\Omega=[0,\frac 1 4]\times[\frac 1 8,\frac 3 8]\times [0,\frac 1 4]$ and the terminal time $T=0.4$. The exact solution is given by $u(t, x, y, z) = e^{-4\pi^2 t}\sin(4\pi(x-\frac 1 4))\sin(4\pi (y-\frac 1 8))\sin(4\pi (z-\frac 1 2))$.
\end{example}

On one aspect, for the spatial accuracy tests, we simulate the PEIFE-linear33 scheme with fixed $N_T=N\times M=4\times 128$ and uniformly refined spatial meshes with $N_x \times N_y \times N_z=8 \times 8 \times 8, 16 \times 16 \times 16, 32 \times 32 \times 32$ and $64 \times 64 \times 64$, respectively, so that the temporal step sizes are much finer than spatial mesh sizes. On another aspect, for the temporal accuracy tests, we run EIFE-linear2, PEIFE-linear22 schemes with uniformly refined time step $N_T = N\times M=4\times 4, 4\times 8, 4\times 16, 4 \times 32$ and EIFE-linear3, PEIFE-linear23, PEIFE-linear33 schemes with $N_T = N\times M=4\times 2, 4\times 4, 4\times 8, 4 \times 16$, respectively, fixing with spatial meshes $N_x \times N_y \times N_z=128 \times 128 \times 128$. Similar to the analysis above, we always iterate the prediction-correction procedure for 4 times. All numerical results are presented in Table \ref{tab 4-3}, including the solution errors at time $T$ with respect to $L^2$-norm and corresponding convergence rates. We still observe the roughly second-order spatial convergence in the $L^2$-norm. We can observe the approximate second-order temporal convergence for EIFE-linear2 and PEIFE-linear22 schemes and third-order convergence for EIFE-linear3, PEIFE-linear23 and PEIFE-linear33 schemes in $L^2$-norm, which basically matches the results for Theorem \ref{thm2}.

\begin{table}[htbp]
	\centering
	\caption{Numerical results on the solution errors measured in the $L^2$-norm and corresponding convergence rates for the EIFE-linear2, EIFE-linear3, PEIFE-linear22, PEIFE-linear23 and PEIFE-linear33 schemes in Example \ref{ex1}}
	\begin{tabular}{|ccccc|}
		\hline
		Method & $N_T$ & $N_x \times N_y\times N_z$ & $\big\|U_{h}^{n,(4)}-u(t_n)\big\|_0$ & CR\\
		\hline
		\multicolumn{5}{|c|}{Spatial accuracy tests}\\
		\hline
		\multirow{4}{40pt}{PEIFE-linear33}& $4 \times 128$ & $8 \times 8\times 8$ & 2.3688e-10 & -\\
		& $4 \times 128$ & $16 \times 16 \times 16$ & 6.0567e-11 & 1.97\\
		& $4 \times 128$ & $32 \times 32\times 32$ & 1.5249e-11 & 1.99\\
		& $4 \times 128$ & $64 \times 64\times 64$ & 3.8396e-12 & 1.99\\
		\hline
		\multicolumn{5}{|c|}{Temporal accuracy tests}\\
		\hline
		\multirow{4}{40pt}{EIFE-linear2}& 16& $128 \times 128\times 128$ & 5.0136e-10 & -\\
		&32 & $128 \times 128\times 128$ & 8.9844e-11 & 2.48\\
		&64 & $128 \times 128\times 128$ & 1.9586e-11 & 2.20\\
		&128 & $128 \times 128\times 128$ & 5.2157e-12 & 1.91\\
		\hline
		\multirow{4}{40pt}{EIFE-linear3} & 8 & $128 \times 128\times 128$ & 7.5763e-10 & -\\
		& 16 & $128 \times 128\times 128$ & 4.8986e-11 & 3.95\\
		&32 & $128 \times 128\times 128$ & 3.6365e-12 & 3.75\\
		&64 & $128 \times 128\times 128$ & 5.1925e-13 & 2.81\\
		\hline
		\multirow{4}{40pt}{PEIFE-linear22} & $4\times 4$& $128 \times 128\times 128$ & 5.0136e-10 & -\\
		&$4 \times 8$ & $128 \times 128\times 128$ & 8.9844e-11 & 2.48\\
		&$4\times 16$ & $128 \times 128\times 128$ & 1.9586e-11 & 2.20\\
		&$4\times 32$ & $128 \times 128\times 128$ & 5.2157e-12 & 1.91\\
		\hline
		\multirow{4}{40pt}{PEIFE-linear23} & $4 \times 2$ & $128 \times 128\times 128$ & 7.5763e-10 & -\\
		& $4 \times 4$ &$128 \times 128\times 128$ & 4.8986e-11 & 3.95\\
		&$4\times 8$ & $128 \times 128\times 128$ & 3.6365e-12 & 3.75\\
		&$4\times 16$ & $128 \times 128\times 128$ & 5.1925e-13 & 2.81\\
		\hline
		\multirow{4}{40pt}{PEIFE-linear33} & $4 \times 2$ & $128 \times 128\times 128$ & 7.5763e-10 & -\\
		& $4 \times 4$ &$128 \times 128\times 128$ & 4.8986e-11 & 3.95\\
		&$4\times 8$ & $128 \times 128\times 128$ & 3.6365e-12 & 3.75\\
		&$4\times 16$ & $128 \times 128\times 128$ & 5.1925e-13 & 2.81\\
		\hline
	\end{tabular}
	\label{tab 4-3}
\end{table}

The overall cost per interation of PEIFE-linear method is tested with fixed $N\times M=4\times 16$ and uniform spatial meshes with $N_x \times N_y\times N_z = 16\times 16\times 16, 32 \times 32 \times 32, 64 \times 64 \times 64$ and $128\times 128 \times 128$. Similar to PEIFE-linear33 scheme, the discussion of computing cost of PEIFE-linear22 and PEIFE-linear23 schemes are omitted. We test the running time with PEIFE-linear33 scheme on 4 cores with iterating the prediction-correction procedure for 4 times. Table \ref{tab4-4} shows the average running time (seconds) per iteration for the PEIFE-linear33 scheme and the corresponding growth factors along the refinement of the spatial mesh. The results also show the approximate linear growth of computational cost along with the number of mesh nodes, which coincides with our efficiency analysis of PEIFE-linear method in Section \ref{algorithm-description}.

\begin{table}[htbp]
	\centering
	\caption{The average running time (seconds) per interation under different spatial meshes and corresponding growth factors with respect to the number of mesh nodes for the PEIFE-linear33 scheme in Example \ref{ex1}}
	\begin{tabular}{|cccc|}
		\hline
		$N_x \times N_y\times N_z$ & $N \times M$ & Average running time  & Growth factor\\
		& &cost per step& \\
		\hline
		$16 \times 16 \times 16$ & $4\times 16$ & 11.98 & - \\
		$32 \times 32 \times 32$ & $4 \times 16$ & 65.47 & 0.82\\
		$64 \times 64 \times 64$ & $4 \times 16$ & 346.91 & 0.80\\
		$128 \times 128 \times 128$ & $4 \times 16$ & 2036.72 & 0.85\\
		\hline
	\end{tabular}
	\label{tab4-4}
\end{table}

Moreover, we also compare the performance of all PEIFE-linear methods and EIFE-linear methods. We simulate various PEIFE-linear and EIFE-linear methods and compare the running time (seconds) and numerical error in $L^2$ and $L^{\infty}$ norms. With fixed $N_x\times N_y\times N_z=64\times 64 \times 64$ and $N_T=N\times M=8\times 8$, we iterate the PEIFE-linear method for 2 times in parallel on 8 cores. Table \ref{tab4-5} verifies again that the accuracy of PEIFE-linear method can reach to that of applying EIFE-linear method sequentially with fine step size, but with a speedup of roughly 3. Additionally, the evolutions of the $L^2$ and $L^{\infty}$ errors of various PEIFE-linear methods are shown in Figure \ref{fig5}, which are tested on 8 cores in parallel. The numerical error declines rapidly in the first step and all PEIFE-linear methods can be practically convergent, which is again much smaller than the number of coarse time subintervals (i.e., $N=8$).

\begin{table}
	\centering
	\caption{The running time (seconds) and numerical error with respect to $L^2$ and $L^{\infty}$ norms for various PEIFE-linear and EIFE-linear methods in Example \ref{ex1}}
	\begin{tabular}{|cccc|}
		\hline
		Method & $L^2$-error & $L^{\infty}$-error & running time\\
		\hline
		EIFE-linear2 & 2.2364e-11 & 4.2218e-10 & 24981.61\\
		EIFE-linear3 & 3.3654e-12 & 1.0626e-11 & 26186.52\\
		PEIFE-linear22 & 2.2364e-11 & 4.2218e-10 & 8900.08\\
		PEIFE-linear23 & 3.3654e-12 & 1.0626e-11 & 8635.68\\
		PEIFE-linear33 & 3.3654e-12 & 1.0626e-11 & 9105.40\\
		\hline
	\end{tabular}
	\label{tab4-5}
\end{table}

\begin{figure}[htbp]
	\centering
	\subfigure{
		\centering
		\includegraphics[width = 170pt,height=150pt]{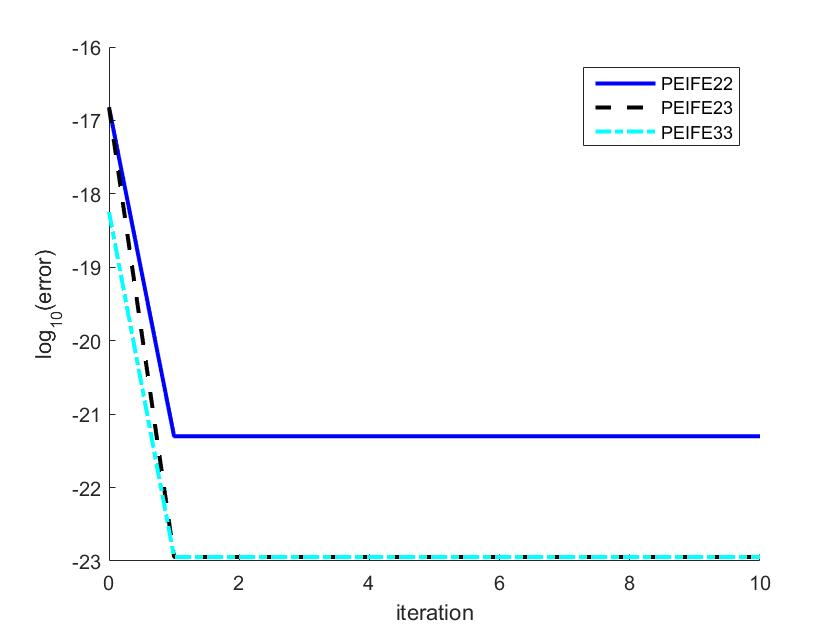}
		\label{fig5-1}
	}
	\subfigure{
		\centering
		\includegraphics[width = 170pt,height=150pt]{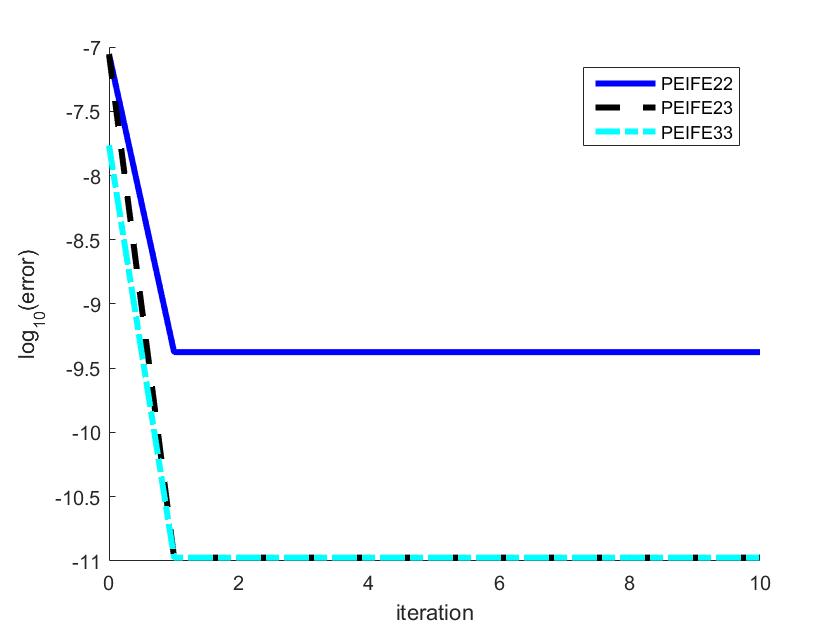}
		\label{fig5-2}
	}
	\caption{The evolutions of $L^2$-error and $L^{\infty}$-error declining curves of the numerical solution produced by PEIFE-linear22, PEIFE-linear23 and PEIFE-linear33 schemes for Example \ref{ex1}.}
	\label{fig5}
\end{figure}

\begin{example}\label{ex7}
	In this example, we consider the following diffusive problem with oscillating source term, which has been studied in \cite{GanderGuttel2013}:
	\begin{equation*}
	\left\{
	\begin{split}
	& u_t = \alpha u_{xx}+g(t, x), \quad x \in [0, 1], \quad 0 \leq t \leq 1,\\
	& u(0, x) = 4x(1-x), \quad x \in [0, 1], \\
	& u(t, 0) = u(t, 1) = 0, \quad 0 \leq t \leq 1,
	\end{split}
	\right.
	\end{equation*}
	where 
	\begin{equation*}
	g(t, x) = h\max\Big\{1-\frac{|c-x|}{w}, 0 \Big\}, \quad c = 0.5+(0.5-w)\sin(2\pi ft).
	\end{equation*}
	The source term $g(t, x)$ is a hat function centered at $c$ with half-width $w = 0.05$ and height $h = 100\alpha^{\frac 1 2}$, oscillating with frequency $f$. 
\end{example}

Fixing the parameter $(\alpha, f)=(0.01, 1),(0.01,10)$ and $(0.1,10)$, we respectively simulate the PEIFE-linear22, PEIFE-linear23 and PEIFE-linear33 schemes to show the numerical results with fixed $N_T=N\times M=4\times 16$ and uniformly refined spatial meshes with $N_x=1024$. The prediction-correction procedure is iterated for 4 times. All numerical results are shown in the Figure \ref{fig8}, which implies that the PEIFE-linear methods for linear equations are also stable for oscillation problems, which is crucial for realistic applications. 

\begin{figure}[htbp]
	\centering
	\subfigure[$\alpha=0.01, f=1$ for PEIFE-linear22 scheme]{
		\centering
		\includegraphics[width = 120pt,height=120pt]{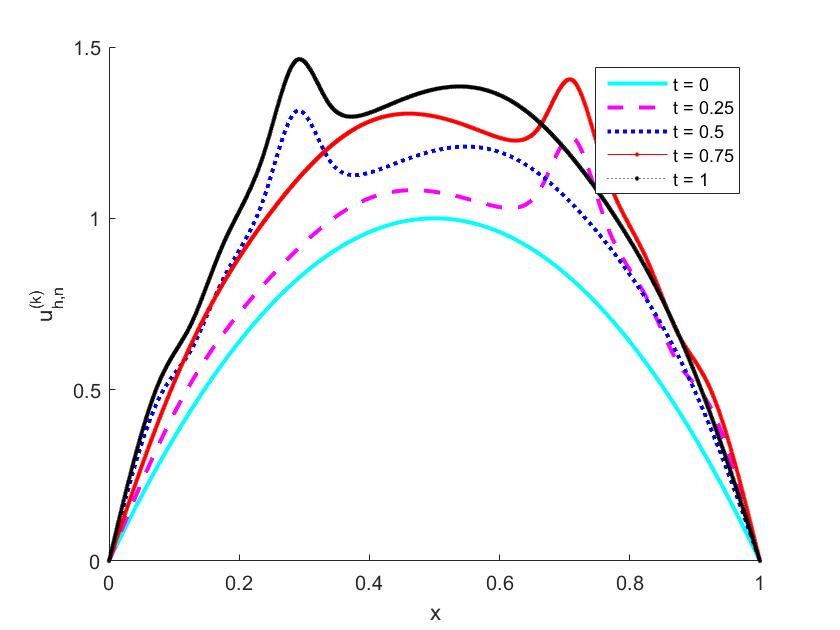}
		\label{fig8-1}
	}
	\subfigure[$\alpha=0.01,f=1$ for PEIFE-linear23 scheme]{
		\centering
		\includegraphics[width = 120pt,height=120pt]{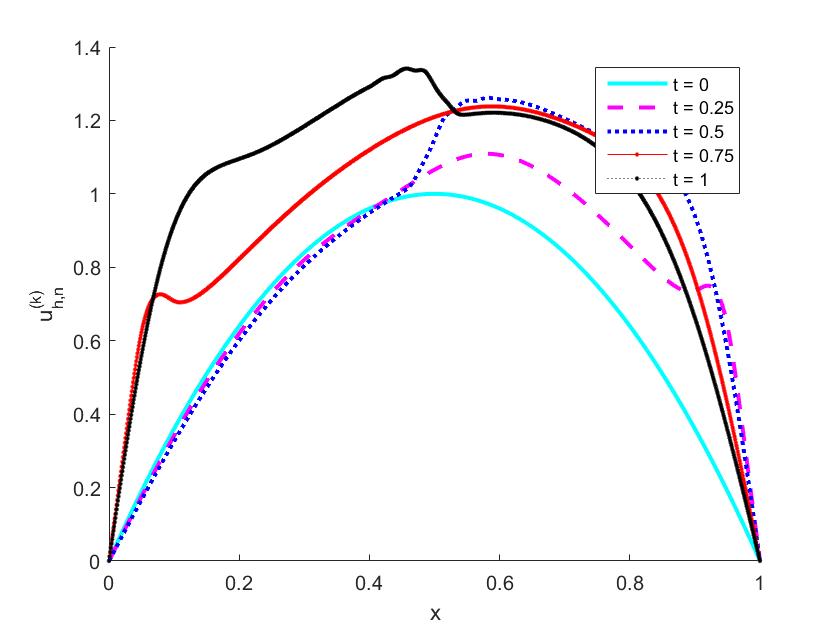}
		\label{fig8-2}
	}
	\subfigure[$\alpha=0.01,f=1$ for PEIFE-linear33 scheme]{
		\centering
		\includegraphics[width = 120pt,height=120pt]{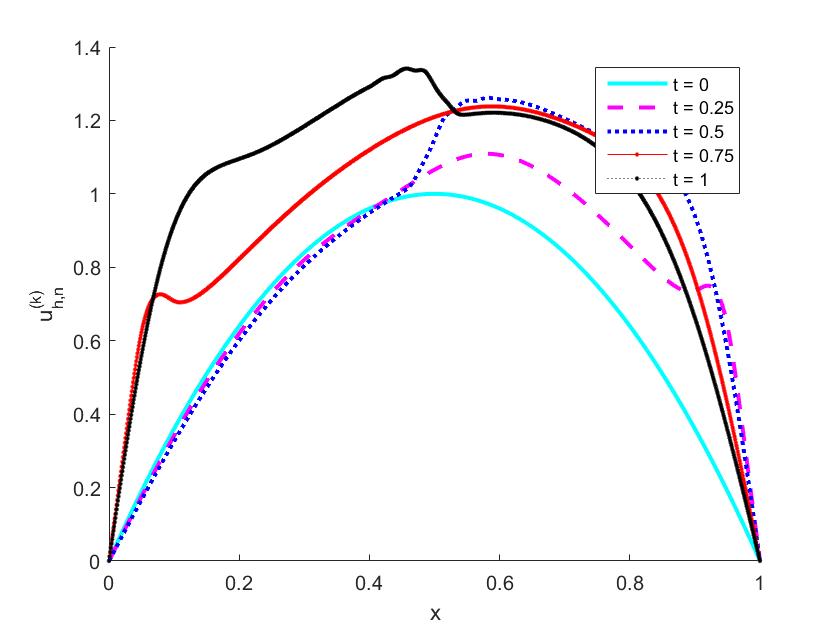}
		\label{fig8-3}
	}

	\subfigure[$\alpha=0.01, f=10$ for PEIFE-linear22 scheme]{
		\centering
		\includegraphics[width = 120pt,height=120pt]{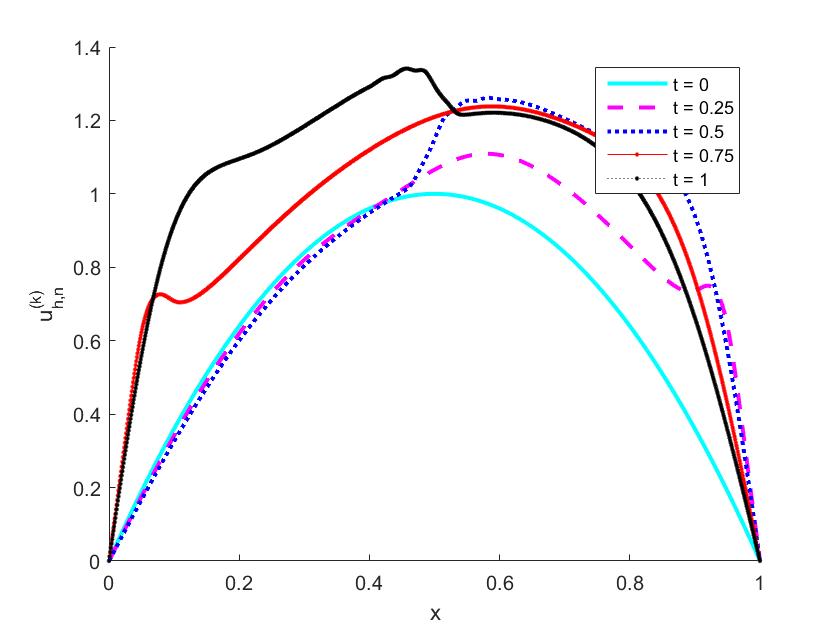}
		\label{fig8-4}
	}
	\subfigure[$\alpha=0.01,f=10$ for PEIFE-linear23 scheme]{
		\centering
		\includegraphics[width = 120pt,height=120pt]{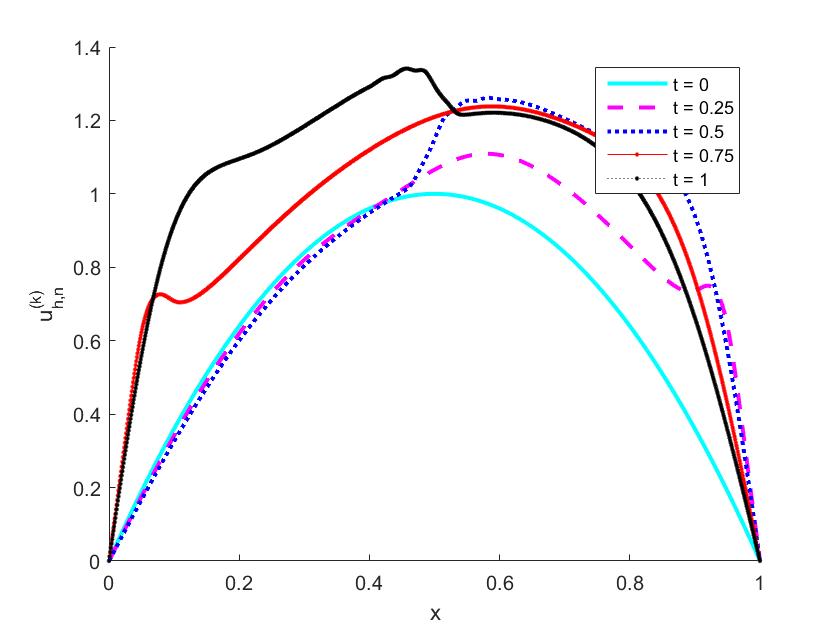}
		\label{fig8-5}
	}
	\subfigure[$\alpha=0.01,f=10$ for PEIFE-linear33 scheme]{
		\centering
		\includegraphics[width = 120pt,height=120pt]{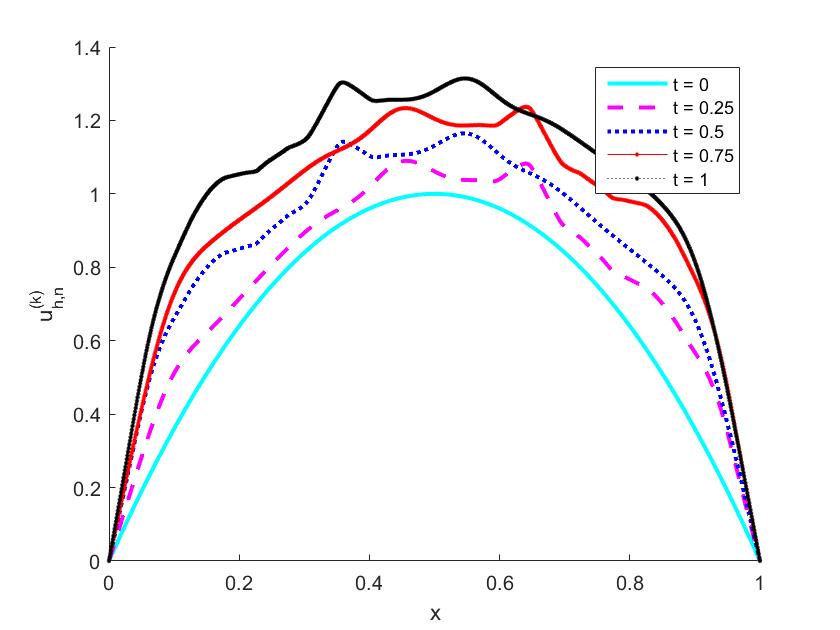}
		\label{fig8-6}
	}

	\subfigure[$\alpha=0.1, f=10$ for PEIFE-linear22 scheme]{
		\centering
		\includegraphics[width = 120pt,height=120pt]{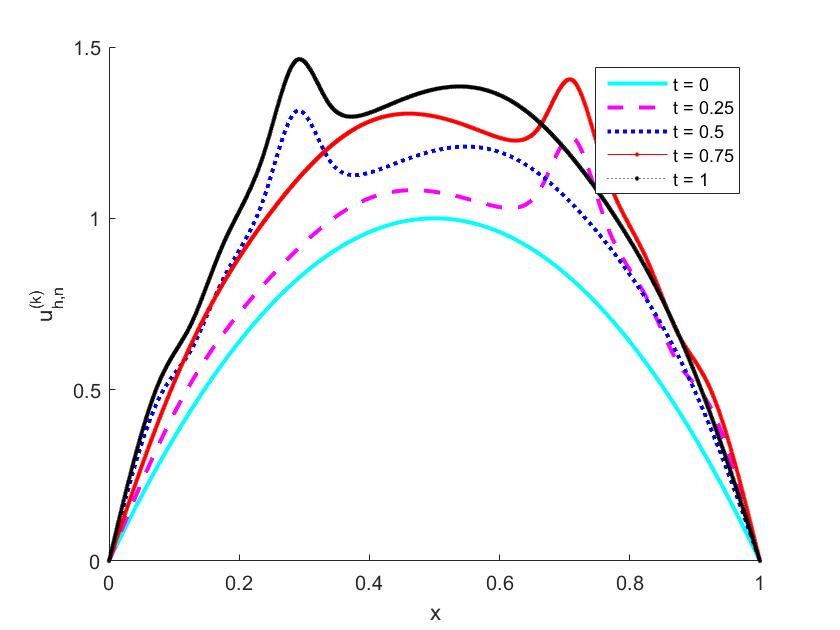}
		\label{fig8-7}
	}
	\subfigure[$\alpha=0.1,f=10$ for PEIFE-linear23 scheme]{
		\centering
		\includegraphics[width = 120pt,height=120pt]{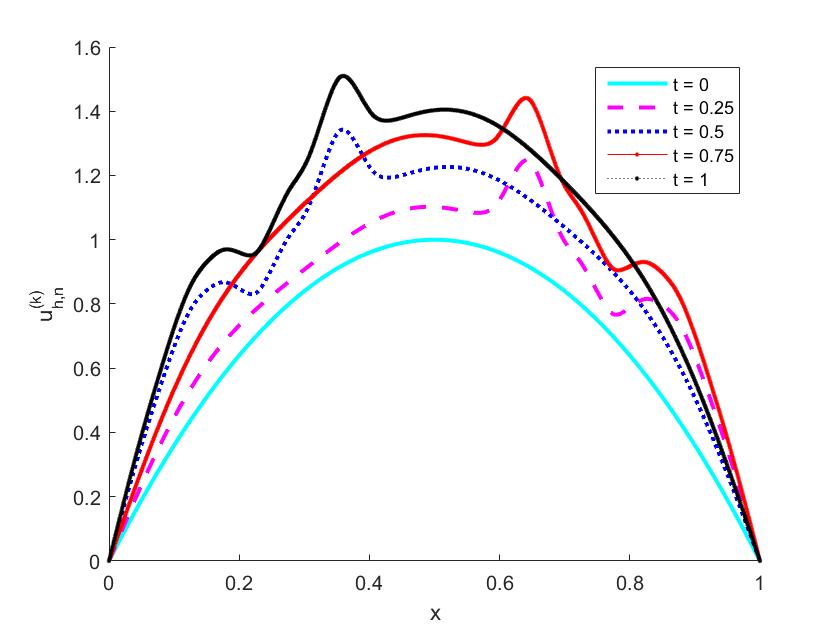}
		\label{fig8-8}
	}
	\subfigure[$\alpha=0.1,f=10$ for PEIFE33 scheme]{
		\centering
		\includegraphics[width = 120pt,height=120pt]{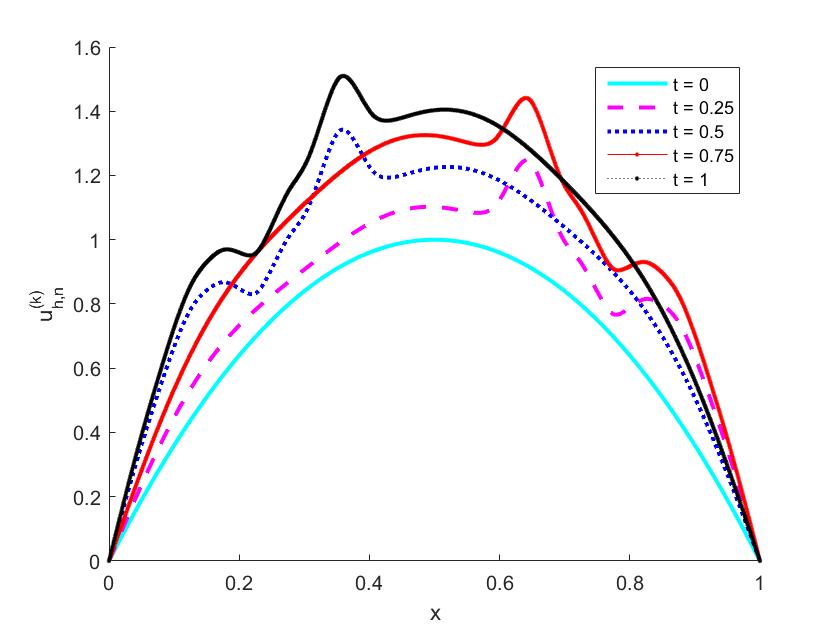}
		\label{fig8-9}
	}
	\caption{Numerical results at $t=0,1.25,0.5,0.75,1$ produced by PEIFE-linear22, PEIFE-linear23, PEIFE-linear33 schemes (from left to right) for Example \ref{ex7} with various $\alpha$ and $f$.}
	\label{fig8}
\end{figure}

\section{Conclusions}\label{conclusion}

For solving a class of linear parabolic equations taking the form \eqref{eq1-1} in regular domains, we have mainly proposed the PEIFE-linear method, which is a modified EIFE-linear method in the parallel-in-time pattern with significant speedup. The proposed PEIFE-linear method is obtained by using finite element method for spatial discretization and then linear exponential Runge-Kutta approximation accompanied with Parareal method for temporal integration. The implementation of PEIFE-linear method allows FFT-based fast calculation techniques and tensor product spectral decomposition. We have derived the explicit error estimates in $L^2$-norm for the PEIFE-linear method after finite iterations. Some numerical examples are also presented to demonstrate the accuracy and high efficiency of the proposed method. Rigorous error analysis of the PEIFE-linear method in $H^1$ or $L^{\infty}$ norm and for the model problem still remains to be explored. 
\bibliographystyle{siam}
\bibliography{ref}

\begin{thebibliography}{10}

\bibitem{AscherRuuth1995}
Uri~M. Ascher, Steven~J. Ruuth, and Brian T.~R. Wetton.
\newblock Implicit-explicit methods for time-dependent partial differential
  equations.
\newblock {\em SIAM J. Numer. Anal.}, 32(3):797--823, 1995.

\bibitem{Guillaume2007}
G.~Bal.
\newblock On the convergence and the stability of the parareal algorithm to
  solve partial differential equations.
\newblock In {\em Domain decomposition methods in science and engineering},
  volume~40 of {\em Lect. Notes Comput. Sci. Eng.}, pages 425--432. Springer,
  Berlin, 2005.

\bibitem{BorislavWill2005}
Will M.~Wright Borislav V.~Minchev.
\newblock {\em A {R}eview of {E}xponential {I}ntegrators for {F}irst {O}rder
  {S}emilinear {P}roblems}.
\newblock preprint. Norwegian University of Science and Technology Trondheim,
  Norway, 2005.

\bibitem{BoydJohn2001}
John~P. Boyd.
\newblock {\em Chebyshev and {F}ourier {S}pectral {M}ethods}.
\newblock Dover Publications, Inc., Mineola, NY, second edition, 2001.

\bibitem{Driscoll2002}
Tobin~A. Driscoll.
\newblock A composite {R}unge-{K}utta method for the spectral solution of
  semilinear {PDE}s.
\newblock {\em J. Comput. Phys.}, 182:357--367, 2002.

\bibitem{DuZhu2004}
Q.~Du and W.~Zhu.
\newblock Stability analysis and application of the exponential time
  differencing schemes.
\newblock {\em J. Comput. Math.}, 22(2):200--209, 2004.

\bibitem{FalgoutFriedhoff2014}
R.~D. Falgout, S.~Friedhoff, Tz.~V. Kolev, S.~P. MacLachlan, and J.~B.
  Schroder.
\newblock Parallel time integration with multigrid.
\newblock {\em SIAM J. Sci. Comput.}, 36(6):C635--C661, 2014.

\bibitem{FarhatChandesris2003}
C.~Farhat and M.~Chandesris.
\newblock Time-decomposed parallel time-integrators: Theory and feasibility
  studies for fluid, structure, and fluid-structure applications.
\newblock {\em Internat. J. Numer. Methods Engrg.}, 58:1397 -- 1434, 11 2003.

\bibitem{FePr2003}
X.~Feng and A.~Prohl.
\newblock Numerical analysis of the {A}llen-{C}ahn equation and approximation
  for mean curvature flows.
\newblock {\em Numerische Mathematik}, 94:33--65, 2003.

\bibitem{GanderLiu2020}
M.~Gander, J.~Liu, S.~Wu, X.~Yue, and T.~Zhou.
\newblock Para{Diag}: Parallel-in-time algorithms based on the diagonalization
  technique.
\newblock 05 2020.

\bibitem{GanderGuttel2013}
M.~J. Gander and S.~G\"{u}ttel.
\newblock P{ARAEXP}: a parallel integrator for linear initial-value problems.
\newblock {\em SIAM J. Sci. Comput.}, 35(2):C123--C142, 2013.

\bibitem{GanderVandewalle2007}
M.~J. Gander and S.~Vandewalle.
\newblock On the superlinear and linear convergence of the parareal algorithm.
\newblock In {\em Domain decomposition methods in science and engineering
  {XVI}}, volume~55 of {\em Lect. Notes Comput. Sci. Eng.}, pages 291--298.
  Springer, Berlin, 2007.

\bibitem{GarridoLee2006}
I.~Garrido, B.~Lee, G.~Fladmark, and M.~Espedal.
\newblock Convergent iterative schemes for time parallelization.
\newblock {\em Math. Comput.}, 75:1403--1428, 07 2006.

\bibitem{HochbruckOstermann2005b}
M.~Hochbruck and A.~Ostermann.
\newblock Explicit exponential {R}unge-{K}utta methods for semilinear parabolic
  problems.
\newblock {\em SIAM J. Numer. Anal.}, 43(3):1069--1090, 2005.

\bibitem{HochbruckOstermann2005a}
M.~Hochbruck and A.~Ostermann.
\newblock Exponential {R}unge-{K}utta methods for parabolic problems.
\newblock {\em Appl. Numer. Math.}, 53(2-4):323--339, 2005.

\bibitem{HochbruckOstermann2010}
M.~Hochbruck and A.~Ostermann.
\newblock Exponential integrators.
\newblock {\em Acta Numer.}, 19:209--286, 2010.

\bibitem{HochbruckOstermann2008}
M.~Hochbruck, A.~Ostermann, and J.~Schweitzer.
\newblock Exponential {R}osenbrock-type methods.
\newblock {\em SIAM J. Numer. Anal.}, 47(1):786--803, 2008/09.

\bibitem{HuangJu2019b}
J.~Huang, L.~Ju, and B.~Wu.
\newblock A fast compact exponential time differencing method for semilinear
  parabolic equations with {N}eumann boundary conditions.
\newblock {\em Appl. Math. Lett.}, 94:257--265, 2019.

\bibitem{HuangJu2019a}
J.~Huang, L.~Ju, and B.~Wu.
\newblock A fast compact time integrator method for a family of general order
  semilinear evolution equations.
\newblock {\em J. Comput. Phys.}, 393:313--336, 2019.

\bibitem{HuangJu2022}
J.~Huang, L.~Ju, and Y.~Xu.
\newblock Efficient exponential integrator finite element method for semilinear
  parabolic equations.
\newblock {\em arXiv preprint}, 2022.

\bibitem{IGG2018}
L.~Isherwood, Z.~J. Grant, and S.~Gottlieb.
\newblock Strong stability preserving integrating factor {R}unge-{K}utta
  methods.
\newblock {\em SIAM J. Numer. Anal.}, 56(6):3276--3307, 2018.

\bibitem{JuLi2021}
L.~Ju, X.~Li, Z.~Qiao, and J.~Yang.
\newblock Maximum bound principle preserving integrating factor {R}unge-{K}utta
  methods for semilinear parabolic equations.
\newblock {\em J. Comput. Phys.}, 439:Paper No. 110405, 18, 2021.

\bibitem{JuZhang2015}
L.~Ju, J.~Zhang, L.~Zhu, and Q.~Du.
\newblock Fast explicit integration factor methods for semilinear parabolic
  equations.
\newblock {\em J. Sci. Comput.}, 62(2):431--455, 2015.

\bibitem{KassamTrefethen2005}
Aly-Khan Kassam and Lloyd~N. Trefethen.
\newblock Fourth-order time-stepping for stiff {PDE}s.
\newblock {\em SIAM J. Sci. Comput.}, 26(4):1214--1233, 2005.

\bibitem{Lawson1967}
J.~Douglas Lawson.
\newblock Generalized {R}unge-{K}utta processes for stable systems with large
  {L}ipschitz constants.
\newblock {\em SIAM J. Numer. Anal.}, 4:372--380, 1967.

\bibitem{LLJF2021}
J.~Li, X.~Li, L.~Ju, and X.~Feng.
\newblock Stabilized integrating factor {R}unge-{K}utta method and
  unconditional preservation of maximum bound principle.
\newblock {\em SIAM J. Sci. Comput.}, 43(3):A1780--A1802, 2021.

\bibitem{LionsMaday2001}
J.-L. Lions, Y.~Maday, and G.~Turinici.
\newblock R\'{e}solution d'{EDP} par un sch\'{e}ma en temps ``parar\'{e}el''.
\newblock {\em C. R. Acad. Sci. Paris S\'{e}r. I Math.}, 332(7):661--668, 2001.

\bibitem{MohebbiDehghan2010}
A.~Mohebbi and M.~Dehghan.
\newblock High-order solution of one-dimensional sine-{G}ordon equation using
  compact finite difference and {DIRKN} methods.
\newblock {\em Math. Comput. Modelling}, 51(5-6):537--549, 2010.

\bibitem{Wu2016}
Wu~S.
\newblock Convergence analysis of the parareal-{E}uler algorithm for systems of
  {ODE}s with complex eigenvalues.
\newblock {\em J. Sci. Comput.}, 67:644--668, 2016.

\bibitem{SanzCalvo1994}
J.~M. Sanz-Serna and M.~P. Calvo.
\newblock {\em Numerical {H}amiltonian {P}roblems}, volume~7 of {\em Applied
  Mathematics and Mathematical Computation}.
\newblock Chapman \& Hall, London, 1994.

\bibitem{ShXuYa2019}
J.~Shen, J.~Xu, and J.~Yang.
\newblock A new class of efficient and robust energy stable schemes for
  gradient flows.
\newblock {\em SIAM Rev.}, 61(3):474--506, 2019.

\bibitem{StaffRonquist2007}
G.~A. Staff and Einar~M. R\o~nquist.
\newblock Stability of the parareal algorithm.
\newblock In {\em Domain decomposition methods in science and engineering},
  volume~40 of {\em Lect. Notes Comput. Sci. Eng.}, pages 449--456. Springer,
  Berlin, 2005.

\bibitem{ThomeeVidar2006}
Vidar Thom\'{e}e.
\newblock {\em {G}alerkin {F}inite {E}lement {M}ethods for {P}arabolic
  {P}roblems}, volume~25.
\newblock Springer-Verlag, Berlin, second edition, 2006.

\bibitem{WangWu2015}
Z.~Wang and S.~Wu.
\newblock Parareal algorithms implemented with {IMEX} {R}unge-{K}utta methods.
\newblock {\em Math. Probl. Eng.}, pages Art. ID 395340, 12, 2015.

\bibitem{WhalenBrio2015}
P.~Whalen, M.~Brio, and J.~V. Moloney.
\newblock Exponential time-differencing with embedded {R}unge-{K}utta adaptive
  step control.
\newblock {\em J. Comput. Phys.}, 280:579--601, 2015.

\bibitem{Wu2014}
S.~Wu.
\newblock Convergence analysis of some second-order parareal algorithms.
\newblock {\em IMA J. Numer. Anal.}, 35(3):1315--1341, 2015.

\bibitem{WuShi2009}
S.~Wu, B.~Shi, and C.~Huang.
\newblock Parareal-{R}ichardson algorithm for solving nonlinear {ODE}s and
  {PDE}s.
\newblock {\em Commun. Comput. Phys.}, pages 883--902, 2009.

\bibitem{Yang2016}
X.~Yang.
\newblock Linear, first and second-order, unconditionally energy stable
  numerical schemes for the phase field model of homopolymer blends.
\newblock {\em J. Comput. Phys.}, 327:294--316, 2016.

\bibitem{ZhuJu2016}
L.~Zhu, L.~Ju, and W.~Zhao.
\newblock Fast high-order compact exponential time differencing {R}unge-{K}utta
  methods for second-order semilinear parabolic equations.
\newblock {\em J. Sci. Comput.}, 67(3):1043--1065, 2016.

\end{thebibliography}

\end{document}